\documentclass[a4paper,reqno,10pt]{amsart}

\usepackage[utf8]{inputenc}
\usepackage{setspace}
\usepackage[top=2.8cm, bottom=2.8cm, left=2.8cm, right=2.8cm]{geometry} 
\usepackage{amssymb,amsmath,dsfont,tikz-cd}
\usepackage{quiver}
\usepackage{mathtools}
\usepackage{thmtools}
\usepackage[shortlabels]{enumitem} 
    \setlist{nolistsep}
    \setlist[enumerate]{label={\upshape(\roman*)}}
\usepackage{hyperref}
\usepackage{cleveref}
\usepackage[all]{xy}
\usepackage{upgreek}
\usepackage{colonequals}
\usepackage{todonotes}
\usepackage{mathdots} 
    \definecolor{urlcolor}{rgb}{0,0,0}
    \definecolor{linkcolor}{rgb}{.7,0.10,0.2}
    \definecolor{citecolor}{rgb}{.12,.54,.11}
\hypersetup{
      breaklinks=true,
      colorlinks=true,
      urlcolor=urlcolor,
      linkcolor=linkcolor,
      citecolor=citecolor,
}

\usepackage[Kac]{dynkin-diagrams} 

\usepackage[backend=biber,style=numeric,sorting=nyt,maxcitenames=99,maxbibnames=100]{biblatex}
\usepackage{mathrsfs}

\usepackage{float} 

\usepackage{xcolor}

\declaretheoremstyle[
    spaceabove=9pt, spacebelow=9pt,
    postheadspace=.5em,
    headfont=\normalfont\bfseries,
    headpunct={},
    headformat={\NUMBER.\@\NOTE},
    notefont=\normalfont\bfseries\boldmath,
    notebraces={}{.},
]{para}

\theoremstyle{plain}

\newtheorem{theorem}{Theorem}[section]

\newtheorem{proposition}[theorem]{Proposition}
\newtheorem{lemma}[theorem]{Lemma}
\newtheorem{conjecture}[theorem]{Conjecture}

\newtheorem{corollary}[theorem]{Corollary}

\theoremstyle{definition}
\newtheorem{remark}[theorem]{Remark}

\newtheorem{question}[theorem]{Question}
\newtheorem*{ACK}{Acknowledgement}
\newtheorem*{NaC}{Notations and conventions}

\theoremstyle{para}

\crefformat{enumi}{{}#2{\upshape#1}#3}
\crefrangeformat{enumi}{{}{\upshape#3#1#4--#5#2#6}}
\crefmultiformat{enumi}{{}{\upshape#2#1#3}}{ and {\upshape#2#1#3}}{, {\upshape#2#1#3}}{, and~{\upshape#2#1#3}}

\crefformat{para}{{}#2{\upshape\S#1}#3}
\crefrangeformat{para}{{}{\upshape#3\S\S#1#4--#5#2#6}}
\crefmultiformat{para}{{}{\upshape#2\S\S#1#3}}{ and {\upshape#2#1#3}}{, {\upshape#2#1#3}}{, and~{\upshape#2#1#3}}

\numberwithin{equation}{section}

\renewcommand{\theta}{\uptheta}
\renewcommand{\alpha}{\upalpha}
\renewcommand{\beta}{\upbeta}
\renewcommand{\gamma}{\upgamma}
\renewcommand{\delta}{\updelta}
\renewcommand{\zeta}{\upzeta}
\renewcommand{\pi}{\uppi\hspace{0.05em}}
\renewcommand{\rho}{\uprho}
\renewcommand{\xi}{\upxi}
\renewcommand{\chi}{\upchi}
\renewcommand{\sigma}{\upsigma}
\renewcommand{\Lambda}{\Uplambda}
\renewcommand{\Gamma}{\Upgamma}
\renewcommand{\phi}{\upphi}
\renewcommand{\psi}{\uppsi}
\renewcommand{\nu}{\upnu}
\renewcommand{\tau}{\uptau}
\renewcommand{\mu}{\upmu}
\renewcommand{\eta}{\upeta}

\newcommand{\CB}{{\mathcal B}}

\newcommand{\CF}{{\mathcal F}}

\newcommand{\CH}{{\mathcal H}}

\newcommand{\CL}{{\mathcal L}}

\newcommand{\CO}{{\mathcal O}}

\newcommand{\CU}{{\mathcal U}}

\newcommand{\CX}{{\mathcal X}}
\newcommand{\CY}{{\mathcal Y}}
\newcommand{\CZ}{{\mathcal Z}}

\DeclareMathOperator{\Hom}{Hom}

\DeclareMathOperator{\MMHM}{\mathsf{MMHM}}

\DeclareMathOperator{\Perv}{\mathsf{Perv}}
\DeclareMathOperator{\MHM}{\mathsf{MHM}}

\DeclareMathOperator{\Map}{Map}

\DeclareMathOperator{\Supp}{Supp}

\DeclareMathOperator{\Crit}{Crit}

\DeclareMathOperator*{\colim}{colim}

\DeclareMathOperator{\pr}{pr}

\newcommand{\tor}{\mathrm{tor}}

\DeclareMathOperator{\fib}{fib}

\DeclareMathOperator{\cl}{cl}

\title[Multiplicative dimensional reduction]{Multiplicative dimensional reduction}

\author{Tasuki Kinjo}

\address{Research Institute for Mathematical Sciences, Kyoto University, Kyoto 606-8502, Japan}

\email{tkinjo@kurims.kyoto-u.ac.jp}

\begin{document}

\setstretch{1.3} 

\begin{abstract}
    We prove the multiplicative version of the dimensional reduction theorem in cohomological Donaldson--Thomas theory. 
More precisely, we show that the BPS cohomology associated with the loop stack of a \(0\)-shifted symplectic stack admits a description analogous to orbifold cohomology, even though our stacks are not necessarily Deligne--Mumford. 
As an application, we propose a new, purely two-dimensional formulation of the topological mirror symmetry conjecture for the moduli space of \(G\)-Higgs bundles, which in turn leads to a formulation of the conjecture for logarithmic \(G\)-Higgs bundles. 
We also investigate a twisted version of the multiplicative dimensional reduction, which applies, in particular, to the cohomological Donaldson--Thomas theory for \(S^1\)-bundles over compact oriented surfaces, and more generally to Seifert-fibred \(3\)-manifolds.
    \end{abstract}
    
    \maketitle
    
    \setcounter{tocdepth}{1}
    
    \tableofcontents
    
    \section{Introduction}
    
    \subsection{Motivations}
    
    In recent years, the theory of Donaldson--Thomas (DT) perverse sheaves $\varphi = \varphi_{\CX}$ for oriented $(-1)$-shifted symplectic stacks $\CX$ introduced by \textcite{BBBBJ15} has been applied to several areas of mathematics beyond DT theory, including three-dimensional topology \cite{abouzaid2020sheaf,kaubrys2024cohomological}, deformation quantization \cite{gunningham2023deformation}, and relative Langlands program \cite{KKPS}.
    Although DT perverse sheaves are difficult to compute in general, there is a particularly simple description for the $(-1)$-shifted cotangent stacks $\mathcal{X} = \mathrm{T}^*[-1]\CY$: namely, the following \emph{dimensional reduction isomorphism} holds \cite{Kin21, KKPS}:
    \begin{equation}\label{eq-additive-dim-red}
     \mathrm{H}^*(\mathrm{T}^*[-1]\CY, \varphi) \cong \mathrm{H}^{\mathrm{BM}}_{-*+\dim \CY}(\CY).
    \end{equation}
    Here, $\mathrm{H}^{\mathrm{BM}}_{*}(-)$ denotes the Borel--Moore homology.
    This theorem is called the \emph{dimensional reduction theorem} for the following reason. Let $S$ be a smooth algebraic surface and $\mathcal{P}\mathrm{erf}_S^{\mathrm{c}}$ the moduli stack of compactly supported perfect complexes on $S$.
    Then there is a natural isomorphism
    \[
        \mathrm{T}^*[-1] \mathcal{P}\mathrm{erf}_S^{\mathrm{c}} \cong \mathcal{P}\mathrm{erf}_{\mathrm{Tot}_S(K_S)}^{\mathrm{c}}.
    \]
    In particular, the dimensional reduction theorem implies an isomorphism
    \[
     \mathrm{H}^*(\mathcal{P}\mathrm{erf}_{\mathrm{Tot}_S(K_S)}^{\mathrm{c}}, \varphi) 
     \cong 
     \mathrm{H}^{\mathrm{BM}}_{-*+\dim \mathcal{P}\mathrm{erf}_S^{\mathrm{c}}}(\mathcal{P}\mathrm{erf}_S^{\mathrm{c}}).
    \]
    The dimensional reduction isomorphism has many applications, including the study of the $\chi$-independence phenomenon for the moduli stack of $\mathrm{GL}_n$-Higgs bundles \cite{kinjo2024cohomological}.
    
    As mentioned above, DT perverse sheaves have also been applied to low-dimensional topology. 
    In particular, Safronov and Williams \cite{safronov2023batalin} proposed that the cohomology of the DT perverse sheaf on the character stack $\mathcal{L}\mathrm{oc}_G(M)$ of a $3$-manifold $M$ with reductive gauge group $G$ provides a model for the state space of the Kapustin--Witten TQFT at a generic quantum parameter.
    Analogously to the case of Calabi--Yau threefolds, it is natural to apply the dimensional reduction theorem to relate the cohomology of the DT perverse sheaf on $\mathcal{L}\mathrm{oc}_G(\Sigma_g \times S^1)$ with the Borel--Moore homology of the character stack $\mathcal{L}\mathrm{oc}_G(\Sigma_g)$, as the latter has been intensively studied over the past two decades \cite{hausel2008mixed,mellit2025toric}.
    However, we cannot directly apply the dimensional reduction theorem, since $\mathcal{L}\mathrm{oc}_G(\Sigma_g \times S^1)$ is not the $(-1)$-shifted cotangent stack of $\mathcal{L}\mathrm{oc}_G(\Sigma_g)$ but rather its \emph{loop stack}, which is a multiplicative analogue of the $(-1)$-shifted tangent stack.\footnote{Since $\mathcal{L}\mathrm{oc}_G(\Sigma_g)$ is $0$-shifted symplectic, the $(-1)$-shifted tangent and cotangent stacks are canonically identified.}
    
    Motivated by this observation, it is natural to ask the following question:
    \begin{question}\label{ques-multiplicative-dim-red}
    For a $0$-shifted symplectic stack $\mathcal{Y}$, is there a way to describe the cohomology of the DT perverse sheaf on the loop stack
    $\mathcal{L}\mathcal{Y} = \mathrm{Map}(S^1, \mathcal{Y})$ in terms of invariants of $\CY$?
    \end{question}
    
    This question has been frequently raised within the DT community (including by the author) since the publication of \cite{Kin21}. Unfortunately, unlike the additive case \eqref{eq-additive-dim-red}, it appears to have no simple answer, as the fibres of the projection $\mathcal{L} \mathcal{Y} \to \mathcal{Y}$ have highly nontrivial topology.
    
    Recently, the author and collaborators \cite[Section~7]{bu2025cohomology} constructed canonical finite-dimensional subspaces 
    \[
        \mathrm{H}_{\mathrm{BPS}}^*(\mathrm{L} Y) \subset \mathrm{H}^*(\mathcal{L} \mathcal{Y}, \varphi), 
        \quad    
        \mathrm{H}_{\mathrm{BPS}}^*(Y) \subset \mathrm{H}_{-*}^{\mathrm{BM}}(\mathcal{Y}),
    \]
    called the \emph{BPS cohomology}, assuming that a good moduli space $\mathcal{Y} \to Y$ exists.
    Furthermore, in \cite[Theorems~9.3.3 and~9.4.3]{bu2025cohomology} and \cite[Theorem 5.1, Theorem 5.3]{hennecart2025strong}, we constructed\footnote{More generally, decompositions are obtained for symmetric oriented $(-1)$-shifted symplectic stacks with good moduli spaces, under mild assumptions satisfied by all reasonable stacks.} decompositions of the cohomology $\mathrm{H}^*(\mathcal{L} \mathcal{Y}, \varphi)$ and Borel--Moore homology $\mathrm{H}_{-*}^{\mathrm{BM}}(\mathcal{Y})$, whose primary summands are the BPS cohomologies, and whose residual summands are described in terms of the image of BPS cohomology under the cohomological Hall induction — an intrinsic version of parabolic induction introduced in \cite{KPS}.
    In particular, the study of $\mathrm{H}^*(\mathcal{L} \mathcal{Y}, \varphi)$ and $\mathrm{H}_{-*}^{\mathrm{BM}}(\mathcal{Y})$ is completely reduced to that of the BPS cohomologies.
    The main result of this paper answers an analogue of \Cref{ques-multiplicative-dim-red} for the BPS cohomology.
    
    \subsection{Main result}
    
    To state the main result, we recall the definition of the BPS sheaf for $0$-shifted symplectic stacks and their loop stacks.
    For simplicity, we recall only the definition of the zeroth BPS cohomology, which suffices for applications to $\mathcal{L}\mathrm{oc}_G(\Sigma_g \times S^1)$ when $G$ is semisimple.
    
    Let $\mathcal{Y}$ be a $0$-shifted symplectic stack admitting a good moduli space $p \colon \mathcal{Y} \to Y$ in the sense of Alper \cite[Definition~4.1]{Alp13}.
    Consider the following diagrams:
    \[
    \begin{tikzcd}
        {\mathcal{LY}} & {\mathcal{Y}} \\
        \mathrm{L} Y & Y
        \arrow["r", from=1-1, to=1-2]
        \arrow["{\tilde{p}}"', from=1-1, to=2-1]
        \arrow["p", from=1-2, to=2-2]
        \arrow["{\bar{r}}"', from=2-1, to=2-2]
    \end{tikzcd}
    \qquad
    \begin{tikzcd}
        {\hat{\mathcal{Y}}} & {\mathrm{T}^*[-1]\mathcal{Y}} & {\mathcal{Y}} \\
        {\hat{Y}} && {Y.}
        \arrow[equals, from=1-1, to=1-2]
        \arrow["\hat{p}"', from=1-1, to=2-1]
        \arrow["\pi", from=1-2, to=1-3]
        \arrow["p", from=1-3, to=2-3]
        \arrow["{\bar{\pi}}"', from=2-1, to=2-3]
    \end{tikzcd}
    \]
    The vertical arrows are good moduli morphisms.
    By the AKSZ formalism \cite[\S2.1]{PTVV13}, $\mathcal{LY}$ and $\hat{\CY}$ are $(-1)$-shifted symplectic stacks, naturally equipped with orientations.
    Hence, we can consider the DT perverse sheaves
    \[
        \varphi_{\mathcal{LY}} \in \Perv(\CL \CY), \qquad \varphi_{\hat{\CY}} \in \Perv(\hat{\CY}).
    \]
    We define the $0$-th BPS sheaves on $\mathrm{L}Y$ and $\hat{Y}$ as the degree-zero perverse cohomology sheaves:
    \[
     \mathcal{BPS}^{(0)}_{\mathrm{L}Y} \coloneqq {}^p \CH^0(\tilde{p}_* \varphi_{\CL \CY}) \in \Perv(\mathrm{L}Y), 
     \qquad
     \mathcal{BPS}^{(0)}_{\hat{Y}} \coloneqq {}^p \CH^0(\hat{p}_* \varphi_{ \hat{\CY}}) \in \Perv(\hat{Y}). 
    \]
    It is shown in \cite[Theorem~7.2.15]{bu2025cohomology} that $\mathcal{BPS}^{(0)}_{\hat{Y}}$ is supported on the zero section $0_Y \colon Y \hookrightarrow \hat{Y}$. We then set
    \[
        \mathcal{BPS}^{(0)}_{Y} \coloneqq 0_Y^* \mathcal{BPS}^{(0)}_{\hat{Y}} \in \Perv(Y).
    \]
    
    Now, for each positive integer $n \in \mathbb{Z}_{>0}$, define the $n$-torsion loop stack
    \[
        \mathcal{L}_n \mathcal{Y} \coloneqq \mathrm{Map}(\mathrm{B}(\mathbb{Z}/n\mathbb{Z}), \mathcal{Y}).    
    \]
    It admits a good moduli space $p_n \colon \mathcal{L}_n \mathcal{Y} \to \mathrm{L}_nY$, and fits into the commutative diagram:
    \[
    \begin{tikzcd}
        {\mathcal{L}_n \mathcal{Y}} & {\mathcal{L \mathcal{Y}}} \\
        {\mathrm{L}_nY} & {\mathrm{L}Y.}
        \arrow["{\iota_n}", from=1-1, to=1-2]
        \arrow["{p_n}"', from=1-1, to=2-1]
        \arrow["{\tilde{p}}", from=1-2, to=2-2]
        \arrow["{\bar{\iota}_n}"', from=2-1, to=2-2]
    \end{tikzcd}
    \]
    The horizontal arrows are closed immersions induced by the map $S^1 = \mathrm{B}\mathbb{Z} \to \mathrm{B}(\mathbb{Z}/n\mathbb{Z})$.
    If $m$ is a positive integer divisible by $n$, there is a natural map 
    \[
     \iota_{n, m} \colon \mathcal{L}_n \mathcal{Y} \to \mathcal{L}_m \mathcal{Y},
    \]
    which is an isomorphism onto a union of connected components. We define the \emph{torsion loop stack}\footnote{The notion of torsion loop stacks appeared in \cite[Definition~1.1]{fu2025hochschild} for derived Deligne--Mumford stacks to equip the inertial stack with the ``correct'' derived structure. In contrast, for derived Artin stacks, the torsion loop and the ordinary loop stacks have different classical truncations.} as the colimit
    \[
         \mathcal{L}_{\mathrm{tor}} \mathcal{Y} \coloneqq \colim_{n} \mathcal{L}_n \mathcal{Y}.
    \]
    The stack $\mathcal{L}_{\mathrm{tor}} \CY$ admits a good moduli space $p_{\tor} \colon \mathcal{L}_{\mathrm{tor}} \CY \to \mathrm{L}_{\tor} Y$, fitting into the commutative diagram
    \[
    \begin{tikzcd}
        {\mathcal{L}_{\tor} \mathcal{Y}} & {\mathcal{L \mathcal{Y}}} \\
        {\mathrm{L}_{\tor}Y} & {\mathrm{L}Y.}
        \arrow["{\iota}", from=1-1, to=1-2]
        \arrow["{p_{\tor}}"', from=1-1, to=2-1]
        \arrow["{\tilde{p}}", from=1-2, to=2-2]
        \arrow["{\bar{\iota}}"', from=2-1, to=2-2]
    \end{tikzcd}
    \]
    By the AKSZ formalism \cite[\S2.1]{PTVV13}, $\mathcal{L}_{\tor} \mathcal{Y}$ is $0$-shifted symplectic. In particular, there is a BPS sheaf $\mathcal{BPS}^{(0)}_{\mathrm{L}_{\tor} Y}$ on $\mathrm{L}_{\tor} Y$.
    Our main result is the following, which can be thought of as the multiplicative analogue of the additive dimensional reduction theorem \eqref{eq-additive-dim-red}:
    
    \begin{theorem}[= \Cref{thm-main-0}]
    There is a natural isomorphism
    \[
     \mathcal{BPS}^{(0)}_{\mathrm{L}Y} \cong \bar{\iota}_* \mathcal{BPS}^{(0)}_{\mathrm{L}_{\mathrm{tor}} Y}.
    \]
    \end{theorem}
    
    Specializing to $\mathcal{Y} = \mathcal{L}\mathrm{oc}_G(\Sigma)$ for an oriented closed surface $\Sigma$ and a semisimple group $G$, we obtain:
    
    \begin{corollary}
    There is a natural isomorphism
    \[
     \mathrm{H}^*(\mathrm{Loc}_G(\Sigma \times S^1), \mathcal{BPS}^{(0)}_{\mathrm{Loc}_G(\Sigma \times S^1)}) 
     \cong 
     \bigoplus_{[g] \in G_{\mathrm{tor}} / G}
     \mathrm{H}^*(\mathrm{Loc}_{\mathrm{C}_G(g)}(\Sigma), \mathcal{BPS}^{(0)}_{\mathrm{Loc}_{\mathrm{C}_G(g)}(\Sigma)}),
    \]
    where the right-hand side runs over the set of $G$-conjugacy classes of torsion elements.
    \end{corollary}

    In fact, we will prove in \Cref{thm-mult-dimred-loc} that only special torsion elements called quasi-isolated elements contribute to the direct sum.
    The classification of the quasi-isolated elements is obtained by \textcite{bonnafe2005quasi}.

    In \cite[Conjecture~10.3.18]{bu2025cohomology}, the author and collaborators formulated a version of the topological mirror symmetry conjecture for the moduli space of $G$-Higgs bundles using the BPS cohomology of the good moduli space of the loop stack, as motivated by the work of Hausel and Thaddeus \cite{hausel2003mirror}.
    A variant of the above corollary for $G$-Higgs bundles provides a new formulation of the topological mirror symmetry conjecture for both $G$-Higgs bundles and their logarithmic analogues.
    See \Cref{conj-stringy-BPS} and \Cref{conj-stringy-IH} for the details.
    
    We also discuss a twisted version of the multiplicative dimensional reduction theorem, giving a $2$-dimensional description of the BPS cohomology for the character stacks for nontrivial $S^1$-bundles and, more generally, Seifert-fibred $3$-manifolds.
    This yields a $3$-dimensional perspective on a twisted form of the topological mirror symmetry conjecture for $G$-Higgs bundles.
    See \Cref{cor-enough-divisible} and \Cref{conj-Seifert} for the details.
    
    \NaC

    \begin{itemize}
        \item All geometric objects are defined over the complex number field.
        \item All derived Artin stacks are assumed to be $1$-Artin, locally of finite type and have affine diagonal.
    \end{itemize}

    \ACK
    
    The author thanks Andres Fernandez Herrero for teaching the author the notion of elliptic endoscopic data,
    Jin Miyazawa for bringing to his attention the notion of Seifert-fibred $3$-manifolds,
    Hyeonjun Park for teaching the author his join work with Jemin You on shifted symplectic rigidification \cite{ParkYou2025_shiftedSymplecticrigidification}
    and Yaoxiong Wen for for bringing the paper \cite{bonnafe2005quasi} to the author's attention.
    He also thanks Pavel Safronov for useful comments on the paper.
    The author was supported by JSPS KAKENHI Grant Number 25K17229.

\section{Torsion loop stacks}

In this section, we introduce \emph{torsion loop stacks} for derived Artin stacks. These were first studied by \textcite{fu2025hochschild} under the name of orbifold inertia stack.

Let $\CY$ be a derived Artin stack. We define the $n$-torsion loop stack and the (ordinary) loop stack by
\[
    \mathcal{L}_n \mathcal{Y} \coloneqq \mathrm{Map}(\mathrm{B}(\mathbb{Z}/n\mathbb{Z}), \mathcal{Y}), \quad         \mathcal{L}\mathcal{Y} \coloneqq \mathrm{Map}(S^1, \mathcal{Y}) = \mathrm{Map}(\mathrm{B} \mathbb{Z}, \mathcal{Y}).      
\]
Here, we summarize basic properties of these stacks:

\begin{proposition}\label{prop-loop}
\begin{enumerate}
    \item The stacks $\mathcal{L}_n \mathcal{Y}$ and $\mathcal{L} \mathcal{Y}$ are derived Artin stacks locally of finite type. \label{item-loop-Artin}
    \item If $\CY$ is locally of finite presentation, then so are $\mathcal{L}_n \mathcal{Y}$ and $\mathcal{L} \mathcal{Y}$. \label{item-loop-lfp}
    \item The natural map $\iota_n \colon \mathcal{L}_n \mathcal{Y} \to \CL\CY$ is a closed immersion. \label{item-loop-cl-imm}
    \item For positive integers $n$ and $m$ with $n \mid m$, the natural transition map $\iota_{n, m} \colon \mathcal{L}_n \mathcal{Y}  \to \mathcal{L}_m \mathcal{Y} $ is an open and closed immersion. \label{item-loop-op-cl}
    \item If $\CY$ admits a good moduli space $p \colon \CY \to Y$, then there are good moduli spaces
          \[
           p_n \colon \CL_n \CY \to \mathrm{L}_n Y, \quad \tilde{p} \colon \CL  \CY \to \mathrm{L}Y.
          \] \label{item-loop-good-moduli}
    \item If $\CY$ is $c$-shifted symplectic, then $\CL_n \CY$ is $c$-shifted symplectic and $\CL \CY$ is $(c-1)$-shifted symplectic. \label{item-loop-AKSZ}
\end{enumerate}
\end{proposition}

\begin{proof}
    The properties \cref{item-loop-Artin} and \cref{item-loop-lfp} for $\CL_n \CY$ are proved in \cite[Theorem 3.3 (1), (2)]{fu2025hochschild}. The corresponding assertions for $\CL \CY = \CY \times_{\CY \times \CY} \CY$ are obvious.
    The property \cref{item-loop-cl-imm} is proved in \cite[Theorem 3.3 (1)]{fu2025hochschild}.
    The property \cref{item-loop-op-cl} is proved in \cite[Theorem 3.3 (3)]{fu2025hochschild}.
    To prove the property \cref{item-loop-good-moduli}, using \cref{item-loop-cl-imm} and \cite[Lemma 4.14]{Alp13}, it suffices to prove the claim for $\CL \CY$.
    Since the map $\CL \CY \to \CY$ is affine by our assumption on stacks, the claim follows from \cite[Lemma 4.14]{Alp13}.
    The property \cref{item-loop-AKSZ} follows from the AKSZ formalism \cite[\S 2.1]{PTVV13}, since $\mathrm{B} (\mathbb{Z} / n \mathbb{Z})$ and $S^1$ are equipped with $\CO$-orientations of dimension $0$ and $1$ respectively.
\end{proof}

Using the transition maps $\iota_{n, m} \colon \CL_n \CY \to \CL_m \CY$ in \Cref{prop-loop} \cref{item-loop-op-cl}, we define the \emph{torsion loop stack}
\[
 \CL_{\tor} \CY \coloneqq \colim_n \CL_n \CY.    
\]
\Cref{prop-loop} implies the following properties of the torsion loop stacks:

\begin{corollary}
    \begin{enumerate}
        \item The stack $\CL_{\tor} \CY $ is a derived Artin stack locally of finite type.
        \item If $\CY$ is locally of finite presentation, then so is $\CL_{\tor} \CY$.
        \item  If $\CY$ admits a good moduli space $p \colon \CY \to Y$, we have the following commutative diagram
        \[
            \begin{tikzcd}
                {\mathcal{L}_{\tor} \mathcal{Y}} & {\mathcal{L \mathcal{Y}}} \\
                {\mathrm{L}_{\tor}Y} & {\mathrm{L}Y}
                \arrow["{\iota}", from=1-1, to=1-2]
                \arrow["{p_{\tor}}"', from=1-1, to=2-1]
                \arrow["{\tilde{p}}", from=1-2, to=2-2]
                \arrow["{\bar{\iota}}"', from=2-1, to=2-2]
            \end{tikzcd}
        \]
        where the vertical arrows are good moduli morphisms and the horizontal arrows are union of closed immersions.
        \item If $\CY$ is $c$-shifted symplectic, then so is $\CL_{\tor} \CY$.
        \end{enumerate}
\end{corollary}

For later use, we explicitly describe the torsion loop stack and the loop stack for quotient stacks:

\begin{lemma}\label{lem-Ltor-quotient}
    Let $R$ be a separated algebraic space equipped with an action of a reductive group $G$, and set $\CY = R/G$.
    \begin{enumerate}
        \item For a positive integer $n \in \mathbb{Z}_{>0}$, we have an identification
    \[
     (\CL_n \CY)_{\cl} \cong \coprod_{\substack{g \in G \\ g^n = e}} (R^{g})_{\cl} / \mathrm{C}_G(g) 
    \]
    where the disjoint union is over the set of $G$-conjugacy classes of $n$-torsion elements. \label{item-loopquotient-1}
    \item There exists a natural identification
    \[
     (\CL \CY)_{\cl} \cong \{ (r, g) \in R \times G \mid g \cdot r = r \} / G.
    \]  \label{item-loopquotient-2}
    \item 
    The natural map $(\CL_n \CY)_{\cl} \hookrightarrow (\CL \CY)_{\cl}$ over the component corresponding to $g \in G$ is given by
    \[
        (R^{g})_{\cl} / \mathrm{C}_G(g)  \cong \left( (R^{g})_{\cl} / \mathrm{C}_G(g) \right) \times \{ g \} \hookrightarrow  \{ (r, g) \in R \times G \mid g \cdot r = r \}.
    \] \label{item-loopquotient-3}
\end{enumerate}
\end{lemma}

\begin{proof}
    Since the proofs of \cref{item-loopquotient-2} and \cref{item-loopquotient-3} are obvious, we only prove \cref{item-loopquotient-1}.
    We first deal with the case when $R$ is a point.
    First, by adjunction, we have a natural morphism
    \[
      \Psi \colon \coprod_{\substack{ g \in G \\ g^n = e}} \mathrm{B} \mathrm{C}_G(g) \to \Map(\mathrm{B}(\mathbb{Z} / n \mathbb{Z}), \mathrm{B} G).
    \]
    We claim that this map is an isomorphism.
    First, it follows from the computation of the cotangent complex of the right-hand side in \cite[Theorem~3.3 (2)]{fu2025hochschild} that $\Map(\mathrm{B}(\mathbb{Z} / n \mathbb{Z}), \mathrm{B} G)$ is a disjoint union of classifying stacks of algebraic groups.
    Take a point $[g] \in \Map(\mathrm{B}(\mathbb{Z} / n \mathbb{Z}), \mathrm{B} G)$ corresponding to an element $g \in G$ with $g^n = e$.
    Then it is clear that the automorphism group of the point $[g]$ is the centralizer group $\mathrm{C}_G(g)$.
    Since each point in $\Map(\mathrm{B}(\mathbb{Z} / n \mathbb{Z}), \mathrm{B} G)$ corresponds to such an element in $G$, we conclude that $\Psi$ is an isomorphism.

    Next, we describe the classical truncation of the mapping stack $\Map(\mathrm{B}(\mathbb{Z} / n \mathbb{Z}), R/G)$.
    Consider the natural map
    \[
       q \colon \Map(\mathrm{B}(\mathbb{Z} / n \mathbb{Z}), R/G) \to \Map(\mathrm{B}(\mathbb{Z} / n \mathbb{Z}), \mathrm{B}G).
    \]
    For each point $[g] \in \Map(\mathrm{B}(\mathbb{Z} / n \mathbb{Z}), \mathrm{B}G)$ corresponding to an element $g \in G$, the fibre $q^{-1}([g])$ is identified with the fixed locus $R^{g}$ inside $R$.
    Hence we obtain the desired statement.
\end{proof}

\begin{corollary}\label{cor-torsionloop-HG}
    Let $C$ be a smooth projective curve with $\Sigma$ the underlying topological $2$-manifold and $G$ be a reductive group.
    Let $\CH_G$ be the moduli stack of $G$-Higgs bundles on $C$ and $\mathcal{L}\mathrm{oc}_G(\Sigma)$ be the moduli stack of $G$-local systems on $\Sigma$.
    Then there exists an isomorphism
    \[
     \CL_{\mathrm{tor}}  \CH_G \cong \coprod_{[g] \in  G_{\mathrm{tor}} / G} \mathcal{H}_{\mathrm{C}_G(g)}, \quad \CL_{\mathrm{tor}}  \mathcal{L}\mathrm{oc}_G(\Sigma) \cong \coprod_{[g] \in G_{\mathrm{tor}} / G} \mathcal{L}\mathrm{oc}_{\mathrm{C}_G(g)}(\Sigma).
    \]
    Here, we set $G_{\mathrm{tor}} \coloneqq \{g \in G \mid g^N = e  \textnormal{ for some $N \in \mathbb{Z}_{>0}$} \}$ where $G$ acts on $G_{\tor}$ by conjugation.
\end{corollary}


We now discuss a version of the above theorem for the moduli stack of semistable $G$-Higgs bundles, denoted as $\CH_{G}^{\mathrm{ss}}$.
For this, we recall some results from \textcite{herrero2025stability}.
For a reductive group $G$ which is not necessarily connected, we define the set of rational cocharacters and characters by
\[
 \Gamma^{\mathbb{Q}}(G) \coloneqq \Hom(\mathbb{G}_{\mathrm{m}}, G)_{\mathbb{Q}}, \quad   \Gamma_{\mathbb{Q}}(G) \coloneqq \Hom(G, \mathbb{G}_{\mathrm{m}})_{\mathbb{Q}}.
\]
Here, the rationalization for the set of cocharacters makes sense even though it is not an abelian group since it is a $\mathbb{Z}$-set: see \cite[Notations and conventions]{herrero2025stability}.
A rational cocharacter $\lambda \in \Gamma^{\mathbb{Q}}(G)$ is called \emph{central} if it lies in $\Gamma^{\mathbb{Q}}(\mathrm{Z}(G))$.
It is shown in \cite[Theorem 2.3.2]{herrero2025stability} that there is a natural isomorphism $\Gamma_{\mathbb{Q}}(G) \cong \Gamma^{\mathbb{Q}}(\mathrm{Z}(G))^{\vee}$.

We introduce the set of rational degrees $\pi_1(G)_{\mathbb{Q}}$ following \cite[\S 4.2.4]{herrero2025stability}.
When $G$ is connected, it is simply defined as the rationalization of the fundamental group, or equivalently the dual of $\Gamma_{\mathbb{Q}}(G)$.
For disconnected groups, $\pi_1(G)_{\mathbb{Q}}$ is defined as the set of $\pi_0(G)$-conjugacy classes of pairs $(F, d)$ of a subgroup 
$F \subset \pi_0(G)$ and $d \in \Gamma_{\mathbb{Q}}(G^{F})$, where $G^{F}$ is the inverse image of $F$ under $G \to \pi_0(G)$.
The set of rational degrees is covariantly functorial with respect to group homomorphisms.
A rational degree $(F, d) \in \pi_1(G)_{\mathbb{Q}}$ is called \emph{zero-type} if $d = 0$.

In \cite[\S 4.2.5]{herrero2025stability}, it is explained that there is a map
\[
 \CH_G \to \pi_1(G)_{\mathbb{Q}},    
\]
which assigns to a $G$-Higgs bundle its  rational degree, compatibly with induction along group homomorphisms.
We define open substacks
\[
    \CH_{G, \textnormal{$0$-type}} \subset \CH_{G}, \quad \CH_{G, \textnormal{$0$-type}}^{\mathrm{ss}} \subset \CH_{G}^{\mathrm{ss}}
\]
consisting of $G$-Higgs bundles with zero-type rational degrees.
When $G^{\circ}$ is semisimple, the above inclusions are in fact equal.

With these preparations, we can state a semistable version of \Cref{cor-torsionloop-HG}.

\begin{proposition}\label{prop-torsionloop-HGss}
    Let $C$ be a smooth projective curve and $G$ be a reductive group.
    Then there is a natural isomorphism
    \[
        \CL_{\mathrm{tor}}  \CH_{G, \textnormal{$0$-type}}^{\mathrm{ss}} \cong \coprod_{[g] \in  G_{\mathrm{tor}} / G} \mathcal{H}_{\mathrm{C}_G(g), \textnormal{$0$-type}}^{\mathrm{ss}}.
    \]    
    In particular, if $G^{\circ}$ is semisimple, we have
    \[
        \CL_{\mathrm{tor}}  \CH_{G}^{\mathrm{ss}} \cong \coprod_{[g] \in  G_{\mathrm{tor}} / G} \mathcal{H}_{\mathrm{C}_G(g), \textnormal{$0$-type}}^{\mathrm{ss}}.
    \]
\end{proposition}

\begin{proof}
    Using \Cref{cor-torsionloop-HG}, it suffices to prove the following three claims for principal $\mathrm{C}_G(g)$-bundles $E$:
    \begin{enumerate}
        \item  $E$ is zero-type if and only if its induction $E \times^{\mathrm{C}_{G}(g)} G$ is zero-type. \label{item-zero-type}
        \item  If $E$ is zero-type and semistable, then its induction $E \times^{\mathrm{C}_{G}(g)} G$ is semistable. \label{item-induction-mo-so}
        \item  If the induction $E \times^{\mathrm{C}_{G}(g)} G$ is semistable, then $E$ is semistable. \label{item-induction-kara-shitagau}
    \end{enumerate}
    The claim \Cref{item-zero-type} is obvious since the map $\mathrm{C}_G(g) \to G$ is injective.
    To prove \Cref{item-induction-mo-so}, recall from \cite[\S 4.2.9]{herrero2025stability} that a rational degree $[F, d] \in \pi_1(\mathrm{C}_G(g))_{\mathbb{Q}}$ is said to be \emph{adapted} if the image of $d$ under the map
    \[
     \Gamma_{\mathbb{Q}}(\mathrm{C}_G(g)^{F})^{\vee} \cong \Gamma^{\mathbb{Q}}(\mathrm{Z}(\mathrm{C}_G(g)^{F})) \to  \Gamma^{\mathbb{Q}}(G^{\iota(F)})
    \]
    is a central cocharacter, where $\iota \colon \pi_0(\mathrm{C}_G(g)) \to \pi_0(G)$ is the natural map.
    Obviously, a zero-type degree for $\mathrm{C}_G(g)$ is adapted.
    It is shown in \cite[Theorem 4.4.1]{herrero2025stability} that the induction of semistable bundles of adapted degrees are again semistable, which implies the desired result.
    The claim \Cref{item-induction-kara-shitagau} follows from \cite[\S 4.4.8]{herrero2025stability} together with the fact that the subspace
    \[
     \mathrm{Lie}(\mathrm{C}_G(g)) \subset \mathrm{Lie}(G)    
    \]
    is a direct summand as a $\mathrm{C}_G(g)$-representation; see also \cite[\S 4.4.9]{herrero2025stability} for a similar argument.

\end{proof}

\section{BPS sheaves}

In this section, we recall the definition of the \emph{BPS sheaves} for $(-1)$-shifted symplectic stacks and $0$-shifted symplectic stacks following \cite[\S 7]{bu2025cohomology}.
For $G$-character stacks of $3$-manifolds, it provides an extension of the perverse sheaf on the stable locus defined by Abouzaid--Manolescu \cite{abouzaid2020sheaf} to the entire moduli space.
For the moduli space of $G$-Higgs bundles and $G$-local systems on surfaces, it provides a modification of the constant sheaf and the intersection complex along the singularities, which enables to extend conjectures and theorems in the non-abelian Hodge theory from the coprimary components for type A groups to all components for arbitrary reductive groups: see \Cref{conj-stringy-BPS} for how it provides a formulation of the topological mirror symmetry conjecture.

Let $\CX$ be a $(-1)$-shifted symplectic stack equipped with an orientation, i.e., a choice of a $\mathbb{Z} / 2\mathbb{Z}$-graded line bundle $\CL$ on $\CX$ together with an isomorphism $o \colon \CL^{\otimes 2} \cong \det(\mathbb{L}_{\mathcal{X}})$.
For such data, \textcite[Theorem 4.8]{BBBBJ15} construct a monodromic mixed Hodge module
\[
 \varphi = \varphi_{\CX}  \in \MMHM(\CX).
\]
by gluing vanishing cycle complexes.
See \cite[\S 5]{bu2025cohomology} for the definition of the monodromic mixed Hodge modules on stacks and for further background.
By abuse of notation, we also denote by $\varphi_{\CX}$ its underlying perverse sheaf.

Now assume that the following conditions are satisfied:
\begin{enumerate}
    \item $\CX$ admits a good moduli space $p \colon \CX \to X$.
    \item $\CX$ is almost symmetric, i.e., for each closed point $x \in \CX$, the action of the neutral component of the stabilizer group $G_x^{\circ}$ on the tangent space $\mathrm{H}^0(\mathbb{T}_{\CX, x})$ is self-dual.
\end{enumerate}
For an integer $c \in \mathbb{Z}_{\geq 0}$, we define the \emph{$c$-th BPS sheaf} by
\[
 \mathcal{BPS}_{X}^{(c)} \coloneqq {}^{\mathrm{p}}{\CH}^{c}(p_* \varphi_{\CX}) \in \MMHM(X).
\]
Here ${}^{\mathrm{p}}{\CH}^{c}(-)$ denotes the $c$-th perverse cohomology sheaf. 
We now state an important vanishing result.

\begin{proposition}[{\cite[Proposition 7.2.9]{bu2025cohomology}}]\label{prop-supp-lem}
    \begin{enumerate}
        \item For $c < 0$, we have $\mathcal{BPS}_{X}^{(c)} = 0$.
        \item Let $x \in \CX$ be a closed point and $\bar{x} \in X$ be its image. Assume that there exists a subtorus $T \subset \mathrm{Z}(G_x^{\circ})$ such that the induced $T$-action on $\mathrm{H}^0(\mathbb{T}_{\CX, x})$ is trivial.
               Then, for $c < \dim T$, we have $\mathcal{BPS}_{X}^{(c)} |_{\bar{x}} = 0$.
    \end{enumerate}
\end{proposition}

Next, we recall the definition of the BPS sheaf associated with a $0$-shifted symplectic stack $\CY$.
To do this, we recall the notion of the central rank from \cite[\S 4.1.3]{bu2025intrinsic}. 
We say that $\CY$ has central rank $\geq n$ if there exists an action of $\mathrm{B} \mathbb{G}_{\mathrm{m}}^n$ on $\CY$ which does not factor through $\mathrm{B} \mathbb{G}_{\mathrm{m}}^{n'}$ for any $n' < n$.
The stack $\CY$ has central rank $n$ if it has central rank $\geq n$ but not central rank $\geq n + 1$.

Now assume that $\CY$ has central rank $c$ and admits a good moduli space $p \colon \CY \to Y$.
Set 
\[\hat{\CY} = \mathrm{T}^*[-1] \CY \coloneqq \mathrm{Tot}_{\CY}(\mathbb{L}_{\CY}[-1])
\]
equipped with a natural $(-1)$-shifted symplectic structure (see \cite[Theorem 2.2]{Cal19}) and standard orientation (see \cite[\S 6.1.4]{bu2025cohomology}).
Consider the following diagram:
\[\begin{tikzcd}
	{\hat{\CY}} & {\mathcal{Y}} \\
	\hat{Y} & {Y}
	\arrow["\pi", from=1-1, to=1-2]
	\arrow["{\tilde{p}}"', from=1-1, to=2-1]
	\arrow["p", from=1-2, to=2-2]
	\arrow["{\bar{\pi}}"', from=2-1, to=2-2]
\end{tikzcd}\]
where the vertical arrows are good moduli morphisms and the horizontal arrows are natural projections.
It is shown in \cite[\S 7.2.14]{bu2025cohomology} that there exists a natural closed immersion
\[
   i \colon \mathbb{A}^{c} \times Y \hookrightarrow \hat{Y}
\]
and that the BPS sheaf $\mathcal{BPS}^{(c)}_{\hat{Y}}$ is of the form
\[
  \mathcal{BPS}_{ \hat{Y}}^{(c)} \cong i_* (\mathcal{IC}_{\mathbb{A}^c} \boxtimes \mathcal{BPS}_Y^{(c)}) 
\]
for some pure Hodge module $\mathcal{BPS}^{(c)}_Y \in \MHM(Y)$. The Hodge module $\mathcal{BPS}^{(c)}_Y$ is called the \emph{$c$-th BPS sheaf} of $\CY$.

\section{Multiplicative dimensional reduction}

\subsection{Multiplicative dimensional reduction: central rank zero}

Here, we prove the multiplicative dimensional reduction in the case of central rank zero.
Let $\CY$ be a $0$-shifted symplectic stack with a good moduli space $p \colon \CY \to Y$.
Consider the loop stack $\CL \CY$ equipped with its natural $(-1)$-shifted symplectic structure.
Let $r \colon \CL \CY \to \CY$ be the natural map.
Then there exists a natural fibre sequence
\[
    r^* \mathbb{L}_{\CY} \to \mathbb{L}_{\CL \CY} \to  r^* \mathbb{L}_{\CY}[1],
\]
which induces a natural trivialization 
\begin{equation}\label{eq-loop-trivialization}
    \det(\mathbb{L}_{\CL \CY}) \cong \CO_{\CL \CY}.
\end{equation}
In particular, $\CL \CY$ is naturally equipped with an orientation, which we call the \emph{standard orientation}.
Our main result is the following:

\begin{theorem}\label{thm-main-0}
    Consider the natural map $\bar{\iota} \colon \mathrm{L}_{\tor} Y \to \mathrm{L}Y$.
    Then there is a natural isomorphism
    \[
        \mathcal{BPS}^{(0)}_{\mathrm{L} Y} \cong \bar{\iota}_* \mathcal{BPS}^{(0)}_{\mathrm{L}_{\mathrm{tor}} Y}.
    \]
\end{theorem}

We first compare the support of the both sides.

\begin{proposition}\label{prop-support}
    Consider the natural map $\bar{\iota} \colon \mathrm{L}_{\tor} Y \to \mathrm{L}Y$.
    The support of $\mathcal{BPS}^{(0)}_{\mathrm{L} Y}$ is contained in the image of $\bar{\iota}$.
\end{proposition}

\begin{proof}
    Take a point $\bar{x} \in \mathrm{L}Y$ which is not contained in the image of $\bar{\iota}$ and let $x \in \CL \CY$ be the corresponding closed point.
    Let $y \in \CY$ be the image of $x$ under the map $\CL \CY \to \CY$. Then it follows from \cite[Lemma 4.3.13]{bu2025cohomology} that $y$ is also closed.
    In particular, the stabilizer group $G_y$ is reductive, and $x$ is identified with an element $g \in G_y$ determined up to conjugacy.
    Since $x$ is closed, the centralizer $\mathrm{C}_{G_y}(g) \cong G_x$ is reductive. Since $g$ is contained in the center of $\mathrm{C}_{G_y}(g)$, it is semisimple.
    Since $x$ is not contained in the image of $\iota$, we see that $g$ is not torsion.
    In particular, the closure
    \[
     \overline{\{  g^n \mid n \in \mathbb{Z} \}} \subset G_y
    \]
    contains a positive-dimensional torus $T' \subset G_y$. 
    Clearly, $T'$ is contained in the centre of $\mathrm{C}_{G_y}(g)$.
    We now claim that the action of $T'$ on $\mathrm{H}^0(\mathbb{T}_{\CL \CY, x})$ is trivial, which implies the proposition by \Cref{prop-supp-lem}.
    Consider the following commutative diagram whose rows are exact:
    \[\begin{tikzcd}
        {\mathrm{H}^{-1}(\mathbb{T}_{\CL \CY, x})} & {\mathrm{H}^{-1}(\mathbb{T}_{\CY, y})} & {\mathrm{H}^{-1}(\mathbb{T}_{\CY, y})} & {\mathrm{H}^{0}(\mathbb{T}_{\CL \CY, x})} & {\mathrm{H}^{0}(\mathbb{T}_{\CY, y})} & {\mathrm{H}^{0}(\mathbb{T}_{\CY, y})} \\
        {\mathfrak{g}_x} & {\mathfrak{g}_y} & {\mathfrak{g}_y}
        \arrow[from=1-1, to=1-2]
        \arrow["\cong"', from=1-1, to=2-1]
        \arrow[from=1-2, to=1-3]
        \arrow["\cong"', from=1-2, to=2-2]
        \arrow["\delta", from=1-3, to=1-4]
        \arrow["\cong"', from=1-3, to=2-3]
        \arrow[from=1-4, to=1-5]
        \arrow["{\cdot g}", from=1-5, to=1-6]
        \arrow[from=2-1, to=2-2]
        \arrow["{\mathrm{Ad}_g}"', from=2-2, to=2-3]
    \end{tikzcd}\]
    The image of $\delta$ is isomorphic to $\mathfrak{g}_x$ as $T'$-representations. In particular, $T'$ acts trivially on the image of $\delta$.
    The cokernel of $\delta$ is annihilated by $g$, and hence also fixed by $T'$. Thus we conclude that $T'$ acts trivially on $\mathrm{H}^0(\mathbb{T}_{\CL \CY, x})$ as desired.

\end{proof}

We now prove \Cref{thm-main-0} over the constant loops:

\begin{proposition}\label{prop-compare-BPS}
    There is a natural isomorphism
    \[
     \mathcal{BPS}_{\mathrm{L}Y}^{(0)} |_Y \cong  \mathcal{BPS}_{\hat{Y}}^{(0)} |_Y.
    \]
\end{proposition}

We will prove this by deforming $\mathrm{L}Y$ to $\hat{Y}$.
To implement this, define derived stacks $\mathfrak{A}$ and $\mathfrak{B}$ by the pushout
\[\begin{tikzcd}
	{\mathrm{pt}} & {\mathbb{A}^1} \\
	{\mathbb{A}^1} & {\mathfrak{A},}
	\arrow["0", from=1-1, to=1-2]
	\arrow["0"', from=1-1, to=2-1]
	\arrow["", from=1-2, to=2-2]
	\arrow[""', from=2-1, to=2-2]
	\arrow["\lrcorner"{anchor=center, pos=0.125, rotate=180}, draw=none, from=2-2, to=1-1]
\end{tikzcd} \quad
\begin{tikzcd}
	{\mathfrak{A}} & {\mathbb{A}^1} \\
	{\mathbb{A}^1} & {\mathfrak{B}}
	\arrow["a", from=1-1, to=1-2]
	\arrow["a"', from=1-1, to=2-1]
	\arrow["i"',from=1-2, to=2-2]
	\arrow["i"',from=2-1, to=2-2]
    \arrow["\lrcorner"{anchor=center, pos=0.125, rotate=180}, draw=none, from=2-2, to=1-1]
\end{tikzcd}\]
where $a \colon \mathfrak{A} \to \mathbb{A}^1$ is the natural map.
Let $b \colon \mathfrak{B} \to \mathbb{A}^1$ be the projection.
Then for $\lambda \in \mathbb{A}^1_{\neq 0}$, $b^{-1}(\lambda) \simeq S^1$ and $b^{-1}(0) \simeq \mathrm{B}\hat{\mathbb{G}}_{\mathrm{a}}$.
By \cite[Proposition B.10.19]{calaque2025aksz}, the stack $\mathfrak{B}$ is $\CO$-compact (see \cite[Definition B.10.16]{calaque2025aksz} for the definition). 
We define the deformed loop stack as the relative mapping stack over $\mathbb{A}^1$
\[
 \tilde{\CL} \CY \coloneqq \mathrm{Map}_{\mathbb{A}^1} (\mathfrak{B}, \CY \times \mathbb{A}^1). 
\]
By construction, $\tilde{\CL} \CY $ fits in the following diagram
\[\begin{tikzcd}
	{\mathrm{T}[-1] \CY} & {\tilde{\CL} \CY} & {\CL \CY \times \mathbb{A}^1_{\neq 0}} \\
	{\{ 0 \}} & {\mathbb{A}^1} & {\mathbb{A}^1_{\neq 0}.}
	\arrow[from=1-1, to=1-2]
	\arrow[from=1-1, to=2-1]
	\arrow["\lrcorner"{anchor=center, pos=0.125}, draw=none, from=1-1, to=2-2]
	\arrow["t", from=1-2, to=2-2]
	\arrow[from=1-3, to=1-2]
	\arrow["\lrcorner"{anchor=center, pos=0.125, rotate=-90}, draw=none, from=1-3, to=2-2]
	\arrow[from=1-3, to=2-3]
	\arrow[from=2-1, to=2-2]
	\arrow[from=2-3, to=2-2]
\end{tikzcd}\]

\begin{lemma}\label{lem-tildeL-good-moduli}
    The stack $\tilde{\CL} \CY$ is a derived Artin stack and admits a good moduli space ${\check{p}} \colon \tilde{\CL} \CY \to \tilde{\mathrm{L}} Y$ which fits in the following Cartesian diagram:
    \[\begin{tikzcd}
        {\hat{Y}} & {\tilde{\mathrm{L}} Y} & {\mathrm{L}Y \times\mathbb{A}^1_{\neq 0}} \\
        {\{ 0 \}} & {\mathbb{A}^1} & {\mathbb{A}^1_{\neq 0}.}
        \arrow[from=1-1, to=1-2]
        \arrow[from=1-1, to=2-1]
        \arrow["\lrcorner"{anchor=center, pos=0.125}, draw=none, from=1-1, to=2-2]
        \arrow["{\bar{t}}", from=1-2, to=2-2]
        \arrow[from=1-3, to=1-2]
        \arrow["\lrcorner"{anchor=center, pos=0.125, rotate=-90}, draw=none, from=1-3, to=2-2]
        \arrow[from=1-3, to=2-3]
        \arrow[from=2-1, to=2-2]
        \arrow[from=2-3, to=2-2]
    \end{tikzcd}\]
\end{lemma}

\begin{proof}
    The only non-trivial point is the existence of the good moduli space $\tilde{\CL} \CY \to \tilde{\mathrm{L}} Y$.
    To prove this, it suffices to show that the map $\tilde{\CL} \CY \to \CY \times \mathbb{A}^1$ is affine by \cite[Lemma 4.14]{Alp13}.
    We prove this using the theory of the Weil restrictions and the deformation to the normal cones.

    First, for a morphism $h \colon \CX_1 \to \CX_2$ of derived stacks, we define the Weil restriction functor
    \[
     h_* \colon \mathsf{dSt_{/ \CX_1}} \to     \mathsf{dSt_{/ \CX_2}}
    \]
    by the right adjoint functor to the base change functor $(\CZ \to \CX_2) \mapsto (\CZ \times_{\CX_2} \CX_1 \to \CX_1)$.
    The deformation space associated with $h$ is
    \[
     \mathcal{D}_{\CX_1 / \CX_2} \coloneqq (\CX_2 \times \{ 0 \} \hookrightarrow \CX_2 \times \mathbb{A}^1)_*  \CX_1 \in \mathsf{dSt_{/ \CX_2 \times \mathbb{A}^1}}.
    \]
    We claim there is an equivalence 
    \begin{equation}\label{eq-deformed-loop-normal-cone}
        \tilde{\CL} \CY \simeq \mathcal{D}_{\CY / \CL \CY}
    \end{equation}
     as derived stacks over $\CY \times \mathbb{A}^1$. 
    This equivalence implies the affineness of $\tilde{\CL} \CY \to \CY \times \mathbb{A}^1$. Indeed, since the map $\CY \to \CL \CY$ is a closed immersion, it follows from \cite[Theorem 12.1.1]{Hekking2025_IntroDAG} that the map $\mathcal{D}_{\CY / \CL \CY} \to \CL \CY \times \mathbb{A}^1$ is affine. Hence the affineness of the map $\CL \CY \to \CY$ implies the desired claim.
    To prove \eqref{eq-deformed-loop-normal-cone}, we first note
    \[
     \mathrm{Map}_{\mathbb{A}^1}(\{ 0 \}, \CY) \simeq  \mathcal{D}_{\CY / \mathrm{pt}}
    \]
    which is nothing but the definition. By using the fact that $\mathcal{D}_{\CY/-}$ preserves limits, we obtain
    \[
        \mathrm{Map}_{\mathbb{A}^1}(\mathfrak{A}, \CY) \simeq \mathcal{D}_{\CY / \CY \times \CY}.    
    \]
    Using the preservation of the limits for $\mathcal{D}_{\CY/-}$ again, we obtain
    \[
        \mathrm{Map}_{\mathbb{A}^1}(\mathfrak{B}, \CY) \simeq \mathcal{D}_{\CY / \CL \CY}    
    \]
    as desired.
\end{proof}

The equivalence \eqref{eq-deformed-loop-normal-cone} together with \cite[Corollary 5.13]{calaque2024shifted} implies that the map $\tilde{\CL} \CY \to \mathbb{A}^1$ carries a relative $(-1)$-shifted symplectic structure $\omega$.
We provide an alternative construction of the $(-1)$-shifted symplectic structure for later use.
Using \cite[Corollary B.5.6]{calaque2024shifted}, we obtain a natural isomorphism
\[
 b_* \CO_{\mathfrak{B}} \cong \CO_{\mathbb{A}^1} \oplus  \CO_{\mathbb{A}^1}[-1].   
\]
The projection $u \colon b_* \CO_{\mathfrak{B}} \to \CO_{\mathbb{A}^1}[-1]$ defines a relative $\CO$-orientation of dimension one.
Indeed, $u |_{\lambda}$ for $\lambda \in \mathbb{A}^1_{\neq 0}$ corresponds to the standard $\CO$-orientation for $S^1$ and $u |_{0}$ corresponds to the $\CO$-orientation on $\mathrm{B} \hat{\mathbb{G}}_{\mathrm{a}} = \mathbb{A}^1[-1]$ used in \cite[Proposition 5.11]{calaque2024shifted}.
The $\CO$-orientation on $b$ together with the relative AKSZ formalism \cite[Theorem 2.34]{calaque2024shifted} implies that the map $\tilde{\CL} \CY \to \mathbb{A}^1$ carries a relative $(-1)$-shifted symplectic structure $\omega$.
It follows from \cite[Proposition 6.1]{khan2025perverse} that $\omega$ canonically upgrades to an exact form.

We now claim that $\tilde{\CL} \CY \to \mathbb{A}^1$ is naturally equipped with an orientation\footnote{Here, we use the term orientation for the square root of the relative canonical bundle.}. More strongly, we will construct a trivialization of the relative canonical bundle 
\begin{equation}\label{eq-trivialize-det-loop}
K_{\tilde{\CL} \CY / \mathbb{A}^1} \coloneqq \det(\mathbb{L}_{\tilde{\CL} \CY / \mathbb{A}^1}) \cong \CO_{\tilde{\CL} \CY}.
\end{equation}
To do this, we will construct a fibre sequence
\[
 \tilde{r}^* \mathbb{T}_{\CY \times \mathbb{A}^1 / \mathbb{A}^1}[-1] \to \mathbb{T}_{\tilde{\CL} \CY / \mathbb{A}^1} \to \tilde{r}^* \mathbb{T}_{\CY \times \mathbb{A}^1 / \mathbb{A}^1}
\]
which implies the desired trivialization of the canonical bundle.
Here, $\tilde{r} \colon \tilde{\CL} \CY \to  \CY \times \mathbb{A}^1$ is the map induced by evaluation $i \colon \mathbb{A}^1 \to \mathfrak{B}$.
For this, take a map $t \colon T \to \tilde{\CL} \CY$ from a derived scheme $T$ over $\mathbb{A}^1$, which corresponds to a map $\hat{t} \colon T \times_{\mathbb{A}^1} \mathfrak{B} \to \CY \times \mathbb{A}^1$.
By using the description of the tangent complex for the mapping stacks (see e.g. \cite[Proof of Theorem 2.5]{PTVV13}),
we obtain
\begin{equation}\label{eq-map-tangent}
 t^* \mathbb{T}_{\tilde{\CL} \CY  / \mathbb{A}^1} \cong b_{T, *} \hat{t}^* \mathbb{T}_{\CY \times \mathbb{A}^1  / \mathbb{A}^1}
\end{equation}
where $b_T \colon T \times_{\mathbb{A}^1} \mathfrak{B} \to T$ are the base change of $b$.
By using \cite[Corollary~B.5.6]{calaque2025aksz}, we obtain an isomorphism
\begin{equation}\label{eq-fibre-description}
    b_{T, *} \hat{t}^* \mathbb{T}_{\CY \times \mathbb{A}^1 / \mathbb{A}^1} \cong \fib(i_T^* \hat{t}^* \mathbb{T}_{\CY \times \mathbb{A}^1 / \mathbb{A}^1} \oplus i_T^* \hat{t}^* \mathbb{T}_{\CY \times \mathbb{A}^1  / \mathbb{A}^1} \to a_{T, *} a_T^* i_T^* \hat{t}^*  \mathbb{T}_{\CY \times \mathbb{A}^1  / \mathbb{A}^1})
\end{equation}
where $i_T \colon T \to T \times_{\mathbb{A}^1} \mathfrak{B}$ and $a_T \colon  T \times_{\mathbb{A}^1} \mathfrak{A} \to T$ is the base change of $i$ and $a$.
Using the projection formula and the base change theorem (see \cite[Theorem B.8.12]{calaque2025aksz}), we obtain an isomorphism
\[
    a_{T, *} a_T^* i_T^* \hat{t}^*  \mathbb{T}_{\CY \times \mathbb{A}^1  / \mathbb{A}^1} \cong i_T^* \hat{t}^*  \mathbb{T}_{\CY \times \mathbb{A}^1  / \mathbb{A}^1} \otimes a_{T, *} \CO_{T \times_{\mathbb{A}^1} \mathfrak{A}} \cong i_T^* \hat{t}^*  \mathbb{T}_{\CY \times \mathbb{A}^1  / \mathbb{A}^1} \otimes t^* a_{ *} \CO_{T \times_{\mathbb{A}^1} \mathfrak{A}} \cong (i_T^* \hat{t}^*  \mathbb{T}_{\CY \times \mathbb{A}^1  / \mathbb{A}^1}) ^{\oplus 2}.
\]
In particular, the isomorphisms \eqref{eq-map-tangent} and \eqref{eq-fibre-description} imply
\[
    t^* \mathbb{T}_{\tilde{\CL} \CY  / \mathbb{A}^1}  = \fib(i_T^* \hat{t}^* \mathbb{T}_{\CY \times \mathbb{A}^1  / \mathbb{A}^1} \to i_T^* \hat{t}^* \mathbb{T}_{\CY \times \mathbb{A}^1  / \mathbb{A}^1} ) \cong \fib(t^* \tilde{r}^* \mathbb{T}_{\CY \times \mathbb{A}^1  / \mathbb{A}^1} \to t^* \tilde{r}^* \mathbb{T}_{\CY \times \mathbb{A}^1   / \mathbb{A}^1}) 
\]
Since it is functorial in $T$, we obtain
\[
    \mathbb{T}_{\tilde{\CL} \CY / \mathbb{A}^1}  \cong  \fib( \tilde{r}^* \mathbb{T}_{\CY \times \mathbb{A}^1  / \mathbb{A}^1} \to \tilde{r}^* \mathbb{T}_{\CY \times \mathbb{A}^1  / \mathbb{A}^1}) 
\]
as desired. One can show that, over $1 \in \mathbb{A}^1$, the endomorphism of $\tilde{r}^* \mathbb{T}_{\CY \times \mathbb{A}^1  / \mathbb{A}^1}$ is given by the difference of the identity morphism and the action morphism,
and over $0 \in \mathbb{A}^1$, it is given by the Atiyah class.

Combining the above discussions, we obtain:

\begin{proposition}\label{prop-relative-symp-ori}
    The natural map $\tilde{r} \colon \tilde{\CL} \CY \to \mathbb{A}^1$ is naturally equipped with a relative exact $(-1)$-shifted symplectic structure and an orientation.
\end{proposition}

By construction, the orientation for $\tilde{r} \colon \tilde{\CL} \CY \to \mathbb{A}^1$ recovers the standard orientation for $\CL \CY$ over $1 \in \mathbb{A}^1$ and the standard orientation for $\hat{\CY}$ over $0 \in \mathbb{A}^1$.

 \Cref{prop-relative-symp-ori} implies that we can define the perverse pullback functor
\[
    \tilde{r}^{\varphi} \colon \mathsf{D^b_c}(\mathbb{A}^1) \to \mathsf{D^b_c}(\tilde{\CL} \CY)
\]
by \cite[Theorem A]{khan2025perverse}.
We define the family version of the DT perverse sheaf on $\tilde{\CL} \CY$ by
\[
 \varphi_{\tilde{\CL} \CY / \mathbb{A}^1} \coloneqq     \tilde{r}^{\varphi} \mathbb{Q}_{\mathbb{A}^1}.
\]
Recall from \Cref{lem-tildeL-good-moduli} that $\tilde{\CL} \CY$ admits a good moduli space $\check{p} \colon \tilde{\CL} \CY  \to \tilde{\mathrm{L}} Y$.
We define a family version of the BPS sheaf by 
\[
 \mathcal{BPS}_{\tilde{\mathrm{L}} Y /\mathbb{A}^1}^{(0)} \coloneqq  {}^{\mathrm{p}}\CH^{1}(\check{p}_* \varphi_{\tilde{\CL} \CY / \mathbb{A}^1} ).  
\]
We will prove \Cref{prop-compare-BPS} by showing that $\mathcal{BPS}_{\tilde{\mathrm{L}} Y /\mathbb{A}^1}^{(0)} |_{Y \times \mathbb{A}^1}$ is constant along the $\mathbb{A}^1$-direction and that restriction to each fibre over $\mathbb{A}^1$ recovers the BPS sheaf of the fibre.
To prove this, we show that the classical truncation of $\mathcal{L} \CY$ together with the natural relative d-critical structure is \'etale locally constant along the $\mathbb{A}^1$-direction.

Consider the map $\ell \colon \mathfrak{B} \to \mathrm{B}\mathbb{G}_{\mathrm{a}} \times \mathbb{A}^1$
which restricts to $1 \mapsto \lambda$ on the stabilizer group at $\lambda \in \mathbb{A}^1_{\neq 0}$ and $\mathrm{B} \hat{\mathbb{G}}_{\mathrm{a}} \hookrightarrow \mathrm{B} \mathbb{G}_{\mathrm{a}}$ over $0 \in \mathbb{A}^1$.
The map $\ell$ induces a map
\[
 \tilde{q} \colon \CL^{\mathrm{u}} \CY \times \mathbb{A}^1 \to \tilde{\CL} \CY,
\]
where $\CL^{\mathrm{u}} \CY \coloneqq \mathrm{Map}(\mathrm{B} \mathbb{G}_{\mathrm{a}}, \CY)$ denotes the unipotent loop stack.
It follows from \cite[Proposition 5.22]{naef2023torsion} that the map $\tilde{q}$ is formally \'etale.
Further, the map $\ell$ relates the $(-1)$-shifted orientation on $\mathfrak{B}$ and the preorientation of $\mathrm{B}\mathbb{G}_{\mathrm{a}} \times \mathbb{A}^1$ given by the projection 
$\Gamma(\mathrm{B}\mathbb{G}_{\mathrm{a}}, \CO_{\mathrm{B}\mathbb{G}_{\mathrm{a}} }) \to \mathbb{C}[-1]$. 
In particular, we see that $\tilde{q}^{\star} \tilde{\omega}$ is constant in the $\mathbb{A}^1$-direction.
We use this to prove the following:

\begin{lemma}\label{lem-locally-isotrivial}
    Let $p \in \hat{\CY} \times \{ 0 \} \subset \tilde{\CL} \CY$ be a closed point.
    Then there exists an Artin stack $\CU$ with a good moduli space $p_{\CU} \colon \CU \to U$ which fits in the following commutative diagram
    \begin{equation}\label{eq-diagram-wanted}
        \begin{tikzcd}
        {(\tilde{\CL} \CY)_{\cl}} & \CU & {\hat{\CY}_{\cl} \times \mathbb{A}^1} \\
        {(\tilde{\mathrm{L}} Y)_{\cl}} & U & {\hat{Y}_{\cl} \times \mathbb{A}^1} \\
        & {\mathbb{A}^1}
        \arrow[from=1-1, to=2-1]
        \arrow["{\eta_1}"', from=1-2, to=1-1]
        \arrow["{\eta_2}", from=1-2, to=1-3]
        \arrow["\lrcorner"{anchor=center, pos=0.125, rotate=-90}, draw=none, from=1-2, to=2-1]
        \arrow[from=1-2, to=2-2]
        \arrow["\lrcorner"{anchor=center, pos=0.125}, draw=none, from=1-2, to=2-3]
        \arrow[from=1-3, to=2-3]
        \arrow["\bar{t}"', from=2-1, to=3-2]
        \arrow["{\bar{\eta}_1}"',from=2-2, to=2-1]
        \arrow["{\bar{\eta}_2}", from=2-2, to=2-3]
        \arrow["\mathrm{pr}_2"', from=2-3, to=3-2]
    \end{tikzcd}\end{equation}
    with the following conditions:
    \begin{itemize}
    \item The horizontal arrows are \'etale.
    \item There exists a closed point $u \in \CU$ such that $\eta_1(u) = p$ and $\eta_{2}(u) = p$ hold.
    \item Let $\tilde{s}$ and $\hat{s}$ denote the natural relative d-critical structures on $(\tilde{\CL} \CY)_{\cl}$ and $\hat{\CY}_{\cl} \times \mathbb{A}^1$ over $\mathbb{A}^1$. Then we have $\eta_1^{\star} \tilde{s} = \eta_2^{\star} \hat{s}$, and the induced orientations are identified.
\end{itemize}
\end{lemma}

\begin{proof}
    Note that the \'etaleness of $\tilde{q} \colon \CL^{\mathrm{u}} \CY \times \mathbb{A}^1 \to \tilde{\CL} \CY$ implies an equivalence of formal completions
    \[
     \tilde{\CL} \CY_{p}^{\wedge} \cong   (\hat{\CY}_{p} \times \mathbb{A}^1)^{\wedge}_p 
    \]
    over $\mathbb{A}^1$ preserving the relative exact $(-1)$-shifted symplectic structures and the stabilizer groups. 
    In particular, the existence of the above diagram \eqref{eq-diagram-wanted} together with the first two conditions follow from \cite[Theorem 4.19]{alper2020luna} and Luna's fundamental lemma \cite[Proposition 4.13]{alper2020luna}.
    Now we prove an equality of relative d-critical structures $\eta_1^{\star} \tilde{s} = \eta_2^{\star} \hat{s}$ after possibly shrinking $\CU$.
    This follows from \cite[Lemma 3.1]{alper2020luna} and \Cref{lem-d-crit-formal}.
    To compare the orientations, let $o_i \colon \CO_{\CU}^{\otimes 2} \cong K_{\CU}$ be the orientations induced from $\eta_i$.
    Let $f_i$ be the corresponding functions over $\CU$ and $\bar{f}_i$ be the induced functions on $U$.
    By replacing $U$ by the double cover parametrizing the square root of $f_1 / f_2$ and $\CU$ be the base change, we conclude that $o_1$ and $o_2$ are isomorphic as desired.
\end{proof}

\begin{lemma}\label{lem-d-crit-formal}
    Let $\CX \to B$ be a morphism between Artin stacks and $p \in \CX$ be a closed point.
    Assume that we are given relative d-critical structures $s_1, s_2 \in \Gamma(\CX, \mathcal{S}_{X / B})$ such that the restriction to the formal completion $s_1 |_{\CX_{p}^{\wedge}}$ and $s_2 |_{\CX_{p}^{\wedge}}$ coincide.
    Then there exists an open neighborhood $\CU$ of $p$ such that $s_1 |_{\CU} = s_2 |_{\CU}$ holds.
\end{lemma}

\begin{proof}
    This is proved by Kaubrys \cite[Lemma 4.8]{kaubrys2024cohomological} when $B = \mathrm{pt}$, and a similar argument works in the relative setting. 
    For the reader's convenience, we will include the proof here.

    By replacing $\CX$ by a smooth cover, we may assume that $\CX$ is a scheme.
    Recall from \cite[\S 4.3]{khan2025perverse} that the relative d-critical structures $s_i$ is given by a function $f_i$ on $\CX$ together with a homotopy $\eta_i \colon d_{\mathrm{dR}} f_i \sim 0$ of relative differential forms over $B$.
    By assumption, $f_1$ and $f_2$ coincides in an open neighborhood of $p$.
    We now compare $\eta_1$ and $\eta_2$. To do this, we may assume $f_1 = f_2$. The homotopy $\eta_1$ and $\eta_2$ determines a loop in the space of $0$-shifted $1$-forms hence $(-1)$-shifted $1$-form $\tau$.
    By assumption, we have $\tau |_{\CX_{p}^{\wedge}} = 0$. Since $\tau$ is a section of a coherent sheaf, we conclude that $\tau$ vanishes over an open neighborhood of $p$ as desired.

\end{proof}

We now use \Cref{lem-locally-isotrivial} to prove \Cref{prop-compare-BPS}.

\begin{proof}[Proof of \Cref{prop-compare-BPS}]
    Let $\bar{t} \colon \tilde{\mathrm{L}} Y \to \mathbb{A}^1$ be the natural map.
    Write $\tilde{\CB} \coloneqq \mathcal{BPS}_{\tilde{\mathrm{L}} Y /\mathbb{A}^1}^{(0)} |_{ Y \times \mathbb{A}^1}$ and consider the specialization map
    \begin{equation}\label{eq-specialization-map}
     \bar{i}^* \tilde{\CB}  \to \psi_{\bar{t}} \tilde{\CB}
    \end{equation}
    where $\bar{i} \colon \mathrm{L} Y \times \{ 0 \} \hookrightarrow \tilde{\mathrm{L}}Y$ is induced by $i \colon \mathcal{L} Y \times \{ 0 \} \hookrightarrow \tilde{\mathcal{L}}Y$.
    We claim that this map yields the desired isomorphism.
    Using \Cref{lem-locally-isotrivial} and the diagram \eqref{eq-diagram-wanted}, we obtain
    \[
    \bar{\eta}_1^* \varphi_{\bar{t}} \tilde{\CB} \cong \varphi_{\bar{t} \circ \bar{\eta}_1} \bar{\eta}_1^* \tilde{\CB}  \cong \varphi_{\pr_2 \circ \bar{\eta}_2} \bar{\eta}_2^*  (\mathcal{BPS}_{\hat{Y}}^{(0)} \boxtimes \mathbb{Q}_{\mathbb{A}^1})\cong \bar{\eta}_2^* \varphi_{\pr_2} (\mathcal{BPS}_{\hat{Y}}^{(0)} \boxtimes \mathbb{Q}_{\mathbb{A}^1}) = 0.
    \]
    In particular, we conclude that $\varphi_{\bar{t}} \tilde{\mathcal{B}} = 0$ and hence the map \eqref{eq-specialization-map} is invertible.
    Since the map $\bar{t} \colon \tilde{\mathrm{L}}Y \to \mathbb{A}^1$ restricts to the projection $\mathrm{L}Y \times \mathbb{A}^1_{\neq 0} \to \mathbb{A}^1_{\neq 0}$, we obtain
    \[
        \psi_{\bar{t}} \tilde{\CB} \cong \mathcal{BPS}_{\mathrm{L} Y}^{(0)} |_Y.
    \]
    We now compute the complex $\bar{i}^* \tilde{\CB}$. 
    Using compatibility of the perverse pullback with finite morphism, we obtain a natural map
    \[
     i^* \varphi_{\tilde{\CL} \CY / \mathbb{A}^1} \to \varphi_{ \hat{\CY}}.
    \]
    Consider the following composite
    \[
        \bar{i}^* \mathcal{BPS}_{\tilde{\mathrm{L}} Y /\mathbb{A}^1}^{(0)}  \to \bar{i}^* \check{p}_* \varphi_{\tilde{\CL} \CY / \mathbb{A}^1} \to \hat{p}_* i^* \varphi_{\tilde{\CL} \CY / \mathbb{A}^1} \to \hat{p}_* \varphi_{ \hat{\CY}}.
    \]
    It follows from \Cref{lem-locally-isotrivial} that $\bar{i}^* \mathcal{BPS}_{\tilde{\mathrm{L}} Y /\mathbb{A}^1}^{(0)}  $ is perverse and the composite induces an isomorphism on the $0$-th perverse cohomology sheaf.
    In particular, we obtain $\bar{i}^* \mathcal{BPS}_{\tilde{\mathrm{L}} Y /\mathbb{A}^1}^{(0)} \cong \mathcal{BPS}_{\hat{Y}}^{(0)}$ and hence $i^* \tilde{\mathcal{B}} \cong \mathcal{BPS}_{\hat{Y}}^{(0)} |_Y$ as desired.

\end{proof}

\Cref{prop-compare-BPS} describes the contribution to the BPS sheaf for the good moduli space of the loop stack from the constant loops.
We describe the contribution from the other components of the support by reducing to the above result.
For this purpose, consider the following map
\begin{equation}\label{eq-eta_n}
 \eta  = \eta_n \colon \CL \CL_n \CY \longrightarrow \CL \CY    
\end{equation}
induced by
\begin{equation}\label{eq-ZZn}
  \mathbb{Z} \longrightarrow \mathbb{Z} \oplus \mathbb{Z} / n   \mathbb{Z}, \quad 1 \mapsto (1, 1).
\end{equation}
By definition, the composition $\CL_n \CY \xrightarrow{\textnormal{const loop}} \CL \CL_n \CY \xrightarrow{\eta} \CL \CY $ is identified with the natural map $\CL_n \CY \to \CL \CY$.
We will use the following properties of $\eta$:

\begin{lemma}\label{lem-eta-properties}
    \begin{enumerate}
        \item The map $\eta$ is \'etale in a neighborhood of the constant loop $\CL_n \CY \hookrightarrow \CL \CL_n \CY$. \label{item-etale}
        \item We equip $\CL \CL_n \CY$ with the $(-1)$-shifted symplectic structure using \Cref{prop-loop} \Cref{item-loop-AKSZ} twice.
              Then $\eta$ preserves the $(-1)$-shifted symplectic structure. \label{item-compare-symplectic}
        \item  Let $\CU \subset \CL \CL_n \CY$ be the open substack where $\eta$ is \'etale.
               Then $\eta |_{\CU}$ preserves the natural orientations. \label{item-compare-orientation}
        
    \end{enumerate}
\end{lemma}

\begin{proof}
    To prove \Cref{item-etale}, it suffices to show that $\eta$ is \'etale at each $\mathbb{C}$-valued point $x \in \CL_n \CY \subset \CL \CL_n \CY$.
    Let $\tilde{x} \colon \mathrm{B} \mathbb{Z} \times \mathrm{B} \mathbb{Z} / n \mathbb{Z} \to \CY$ be the map corresponding to $x$.
    Then we have an identification
    \[
     \mathbb{T}_{\CL \CL_n \CY, x} \cong \Gamma (\mathrm{B} \mathbb{Z} \times \mathrm{B} \mathbb{Z} / n \mathbb{Z}, \tilde{x}^* \mathbb{T}_{\CY}), \quad \mathbb{T}_{\CL \CY, \eta(x)} \cong \Gamma (\mathrm{B} \mathbb{Z}, u^* \tilde{x}^* \mathbb{T}_{\CY})
    \]
    where $u \colon \mathrm{B} \mathbb{Z} \to \mathrm{B} \mathbb{Z}  \times \mathrm{B} \mathbb{Z} / n \mathbb{Z}$ is induced by \eqref{eq-ZZn}.
    Since $x$ is a constant loop, the complex $\tilde{x}^* \mathbb{T}_{\CY}$ is of the form $\CO_{ \mathrm{B} \mathbb{Z}} \boxtimes \CF$ for some $\CF \in \mathsf{Perf}(\mathrm{B} \mathbb{Z} / n \mathbb{Z})$.
    We let $\CF = \bigoplus_{i \in (\mathbb{Z} / n \mathbb{Z})^{\vee}} V_i \otimes \CL_i$ be the weight decomposition.
    Then we have 
    \[
        \Gamma (\mathrm{B} \mathbb{Z} \times \mathrm{B} \mathbb{Z} / n \mathbb{Z}, \tilde{x}^* \mathbb{T}_{\CY}) \cong  V_0 \otimes \Gamma (\mathrm{B} \mathbb{Z} \times \mathrm{B} \mathbb{Z} / n \mathbb{Z}, \CO_{ \mathrm{B} \mathbb{Z} \times \mathrm{B} \mathbb{Z} / n \mathbb{Z}} ) \cong V_0 \otimes \Gamma (\mathrm{B} \mathbb{Z}  , \CO_{ \mathrm{B} \mathbb{Z}  } ).
    \]
    On the other hand, we have 
    \[
        \Gamma (\mathrm{B} \mathbb{Z}, u^* \tilde{x}^* \mathbb{T}_{\CY}) \cong  \bigoplus_i V_i \otimes \Gamma (\mathrm{B} \mathbb{Z}  , \CL_i |_{\mathrm{B} \mathbb{Z}} )  \cong V_0 \otimes \Gamma (\mathrm{B} \mathbb{Z}  , \CO_{ \mathrm{B} \mathbb{Z}  } ). 
    \]
    In particular, we see that the natural map $\mathbb{T}_{\CL \CL_n \CY, x} \to \mathbb{T}_{\CL \CY, \eta(x)}$ is invertible as desired.

    For \Cref{item-compare-symplectic}, note that the $(-1)$-shifted symplectic structure on $\CL \CL_n \CY$ is induced by the $\CO$-orientation on $\mathrm{B} \mathbb{Z} \times \mathrm{B} \mathbb{Z} / n \mathbb{Z}$ given by 
    \[
        \Gamma(\mathrm{B} \mathbb{Z} \times \mathrm{B} \mathbb{Z} / n \mathbb{Z}, {\CO}_{\mathrm{B} \mathbb{Z} \times \mathrm{B} \mathbb{Z} / n \mathbb{Z}}) \cong \Gamma(\mathrm{B} \mathbb{Z}, \CO_{\mathrm{B} \mathbb{Z}}) \to \mathbb{C}[-1].
    \]
    Clearly, the map $u \colon \mathrm{B} \mathbb{Z} \to \mathrm{B} \mathbb{Z} \times \mathrm{B} \mathbb{Z} / n \mathbb{Z}$ preserves the $\CO$-orientation, hence we obtain the desired result.

    For \Cref{item-compare-orientation}, take a point $x \in \CL \CL_n \CY$ at which $\eta$ is \'etale and $\tilde{x} \colon \mathrm{B} \mathbb{Z} \times \mathrm{B} \mathbb{Z} / n \mathbb{Z} \to \CY$ be the corresponding map.
    Consider the weight decomposition $\tilde{x}^* \mathbb{T}_{\CY} = \bigoplus_{i \in (\mathbb{Z} / n \mathbb{Z})^{\vee}} \CF_i \boxtimes \CL_i$.
    By the assumption that $\eta$ is \'etale at $x$ and the vanishing of the global section of $\CL_i$ for $i \neq 0$, we have 
    \[
          \Gamma(\mathrm{B} \mathbb{Z}, \CF_i \otimes \CL_i |_{\mathrm{B} \mathbb{Z}}) = 0
    \]
    for $i \neq 0$, and the map $\mathbb{T}_{\CL \CL_n \CY, x} \to \mathbb{T}_{\CL \CY, \eta(x)} $ is identified with
    \[
     \Gamma(\mathrm{B} \mathbb{Z} \otimes \mathrm{B} \mathbb{Z} / n \mathbb{Z}, \CF_0 \boxtimes \CO_{\mathrm{B} \mathbb{Z} / n \mathbb{Z}}) \xrightarrow{\cong} \Gamma(\mathrm{B} \mathbb{Z} , \CF_0).    
    \]
    Recall that the orientations for $\CL \CL_n \CY$ and $\CL \CY$ are defined by the trivializations
    \begin{equation}\label{eq-isom-tangent}
        \det \Gamma(\mathrm{B} \mathbb{Z} \otimes \mathrm{B} \mathbb{Z} / n \mathbb{Z}, \CF_0 \boxtimes \CO_{\mathrm{B} \mathbb{Z} / n \mathbb{Z}}) \cong \mathbb{C}, \quad \det \Gamma(\mathrm{B} \mathbb{Z} , \CF_0) \cong \mathbb{C}
    \end{equation}
    as constructed in \eqref{eq-trivialize-det-loop}, which in turn is constructed by the fibre sequences
    \begin{equation}\label{eq-two-det-trivialization}
        \Gamma(\mathrm{B} \mathbb{Z} / n \mathbb{Z}, \CF_0 |_{\mathrm{pt}} \otimes \CO_{\mathrm{B} \mathbb{Z} / n \mathbb{Z}}) [-1] \to  \Gamma(\mathrm{B} \mathbb{Z} \otimes \mathrm{B} \mathbb{Z} / n \mathbb{Z}, \CF_0 \boxtimes \CO_{\mathrm{B} \mathbb{Z} / n \mathbb{Z}}) \to \Gamma(\mathrm{B} \mathbb{Z} / n \mathbb{Z}, \CF_0 |_{\mathrm{pt}} \otimes \CO_{\mathrm{B} \mathbb{Z} / n \mathbb{Z}}) 
    \end{equation}
    \[
        \CF_0 |_{\mathrm{pt}}[-1] \to \Gamma(\mathrm{B} \mathbb{Z} , \CF_0) \to \CF_0 |_{\mathrm{pt}}.
    \]
    Since the fibre sequences \eqref{eq-two-det-trivialization} are naturally identified under the isomorphism \eqref{eq-isom-tangent}, we conclude that $\eta$ preserves the orientations over the locus on which $\eta$ is \'etale.
\end{proof}

With these preparations, we can prove \Cref{thm-main-0}.

\begin{proof}[Proof of \Cref{thm-main-0}]
    Thanks to \Cref{prop-support}, it suffices to prove an isomorphism
    \[
     \bar{\iota}^* \mathcal{BPS}^{(0)}_{\mathrm{L} Y} \cong  \mathcal{BPS}^{(0)}_{\mathrm{L}_{\mathrm{tor}} Y}.  
    \]
    In particular, it is enough to prove that for each $n \in \mathbb{Z}_{>0}$ there is a natural isomorphism:
    \[
     \bar{\iota}_n^* \mathcal{BPS}^{(0)}_{\mathrm{L} Y} \cong  \mathcal{BPS}^{(0)}_{\mathrm{L}_{n} Y}.  
    \]
    Recall that $\bar{\iota}_n$ is identified with the following composite
    \[
    \mathrm{L}_n Y \hookrightarrow \mathrm{L} \mathrm{L}_n Y \xrightarrow{\bar{\eta}_n}  \mathrm{L} Y
    \]
    where $\bar{\eta}_n$ is induced from $\eta_n$ in \eqref{eq-eta_n}.
    It follows from \Cref{lem-eta-properties} that we have isomorphisms
    \begin{equation}\label{eq-etale-dakara}
     \mathcal{BPS}_{\mathrm{L} \mathrm{L}_n Y}^{(0)} |_{\mathrm{L}_n Y} \cong \bar{\eta}_n^* \mathcal{BPS}_{\mathrm{L} Y}^{(0)}  |_{\mathrm{L}_n Y} \cong \bar{\iota}_n^* \mathcal{BPS}_{\mathrm{L} Y}^{(0)}.
    \end{equation}
    On the other hand, \Cref{prop-compare-BPS} implies
    \begin{equation}\label{eq-main-atochotto}
        \mathcal{BPS}_{\mathrm{L} \mathrm{L}_n Y}^{(0)} |_{\mathrm{L}_n Y} \cong  \mathcal{BPS}_{\mathrm{L}_n Y}^{(0)}.
    \end{equation}
    By combining \eqref{eq-etale-dakara} and \eqref{eq-main-atochotto}, we obtain the desired statement.
\end{proof}

\begin{remark}\label{rem-hodge}
    The perverse pullback functor \cite{khan2025perverse} can be defined at the level of monodromic mixed Hodge modules,
    and the same argument yields an upgrade of \Cref{thm-main-0} to an isomorphism of mixed Hodge modules.
    Since $ \mathcal{BPS}_{\mathrm{L}_{\mathrm{tor}} Y}^{(0)}$ is pure, we see that $\mathcal{BPS}_{\mathrm{L} Y}^{(0)}$ is also pure.
\end{remark}

\subsection{Multiplicative dimensional reduction: general central rank}

We now discuss a generalization of \Cref{thm-main-0} to the $c$-th BPS sheaves.
We do not fully describe the $c$-th BPS sheaf of the loop stacks, but we describe its cohomology by reducing to \Cref{thm-main-0}.

Let $\CY$ be a $0$-shifted symplectic stack with central rank $c$ admitting a good moduli space $p \colon \CY \to Y$.
We fix a $\mathrm{B} \mathbb{G}_{\mathrm{m}}^c$-action on $\CY$ such that for each $y \in \CY$, the induced map $\mathbb{G}_{\mathrm{m}}^c \to G_y$ is injective.
We assume that $\CY$ admits a global equivariant parameter, namely, there exists a map
\[
 a \colon \CY \to \mathrm{B} \mathbb{G}_{\mathrm{m}}^c    
\]
such that for each $y \in \CY$, the composition $\mathrm{B} \mathbb{G}_{\mathrm{m}}^c \times \{ y \} \to \CY \xrightarrow{a} \mathrm{B} \mathbb{G}_{\mathrm{m}}^c $ is finite.
See \cite[\S 9.1.2]{bu2025cohomology} for generalities on global equivariant parameters. The existence of a global equivariant parameter is a very mild assumption: for example, it is satisfied for global quotient stacks.

As explained in \cite[Proposition 3.1.9]{bu2025cohomology}, there is a moment map with respect to the $\mathrm{B} \mathbb{G}_{\mathrm{m}}^c$-action
\[
 \mu \colon \CY \to (\mathbb{A}^1[-1])^c.    
\]
Further, it will be shown by Park--You \cite{ParkYou2025_shiftedSymplecticrigidification} that there exists a $0$-shifted symplectic structure on 
$\bar{\CY} \coloneqq \mu^{-1}(0) / \mathrm{B} \mathbb{G}_{\mathrm{m}}^c$ such that the following diagram is a Lagrangian correspondence
\[\begin{tikzcd}
	& {\mu^{-1}(0)} \\
	\CY && {\bar{\CY}.}
	\arrow[from=1-2, to=2-1]
	\arrow[from=1-2, to=2-3]
\end{tikzcd}\]
Note that the backward map $\mu^{-1}(0) \to \CY$ induces an isomorphism on the classical truncations.
In particular, there exists a natural map $\CY_{\cl} \to \bar{\CY}_{\cl}$ which is a $\mathbb{G}_{\mathrm{m}}^c$-gerbe.
By passing to the loop stack, we obtain a natural map
\[
 q \colon (\CL \CY)_{\cl} \to (\CL \bar{\CY})_{\cl} 
\]
preserving the d-critical structure and the orientation.
Now consider the following diagram
\[\begin{tikzcd}
	{(\CL \CY)_{\cl}} && {(\CL \bar{\CY})_{\cl}} \\
	{(\mathrm{L} Y)_{\cl}} & {(\mathrm{L} Y)_{\cl} / \mathbb{G}_{\mathrm{m}}^c} & {(\mathrm{L} \bar{Y})_{\cl}.}
	\arrow["q", from=1-1, to=1-3]
	\arrow["{\tilde{p}}"', from=1-1, to=2-1]
	\arrow["\lrcorner"{anchor=center, pos=0.125}, draw=none, from=1-1, to=2-2]
	\arrow["{\tilde{p}'}"', from=1-3, to=2-2]
	\arrow["{\tilde{p}_{\mathrm{rig}}}", from=1-3, to=2-3]
	\arrow["{\bar{q}}"', from=2-1, to=2-2]
	\arrow["g"', from=2-2, to=2-3]
\end{tikzcd}\]
Here, the vertical arrows are good moduli morphisms.
Note that the action of $\mathbb{G}_{\mathrm{m}}^c$ on $(\mathrm{L}Y)_{\mathrm{cl}}$ has finite stabilizer groups.
This follows from the fact that the map $(\mathrm{L} Y)_{\cl} \to \mathbb{G}_{\mathrm{m}}^c$ induced by the global equivariant parameter is $\mathbb{G}_{\mathrm{m}}^c$-equivariant.
In particular, the stack $(\mathrm{L} Y)_{\cl} / \mathbb{G}_{\mathrm{m}}^c$ is Deligne--Mumford.

The following statement combined with \Cref{thm-main-0} can be thought of as a multiplicative version of the dimensional reduction for $c$-th BPS sheaves.

\begin{proposition}\label{prop-mult-dim-red-general}
    There is a natural isomorphism
    \[
    g_* \bar{q}_* \mathcal{BPS}^{(c)}_{\mathrm{L}Y} \cong  \mathcal{BPS}^{(0)}_{\mathrm{L}\bar{Y}} \otimes \mathrm{H}^*(\mathbb{G}_{\mathrm{m}}^c)[c].
    \]
    In particular, we have an isomorphism
    \[
     \mathrm{H}^* (\mathrm{L}Y, \mathcal{BPS}^{(c)}_{\mathrm{L}Y}) \cong \mathrm{H}^*(\mathrm{L}\bar{Y}, \mathcal{BPS}^{(0)}_{\mathrm{L}\bar{Y}} ) \otimes \mathrm{H}^*(\mathbb{G}_{\mathrm{m}}^c)[c].   
    \]
\end{proposition}

\begin{proof}
    Since the map $q$ preserves the d-critical structures and orientations, we have
    \[
     \varphi_{\CL \CY} \cong q^{*}\varphi_{ \CL  \bar{ \CY}}[c].    
    \]
    In particular, we have
    \[
     \tilde{p}_*     \varphi_{\CL \CY} \cong \bar{q}^* \tilde{p}'_*\varphi_{ \CL  \bar{ \CY}}[c].
    \]
    We define $\mathcal{BPS}^{(0)}_{\mathrm{L}Y /\mathbb{G}_{\mathrm{m}}^c} \coloneqq {}^p \CH^0(\tilde{p}'_*\varphi_{ \CL  \bar{ \CY}})$. Then the above isomorphism yields
    \begin{equation}\label{eq-BPS-pullback}
        \mathcal{BPS}^{(c)}_{\mathrm{L}Y} \cong \bar{q}^* \mathcal{BPS}^{(0)}_{\mathrm{L}Y /\mathbb{G}_{\mathrm{m}}^c}[c].
    \end{equation}
    On the other hand, using the fact that $g$ is the good moduli space morphism for Deligne--Mumford stack, we see that $g_*$ is perverse t-exact and
    \begin{equation}\label{eq-BPS-push}
     g_*  \mathcal{BPS}^{(0)}_{ \mathrm{L}Y /\mathbb{G}_{\mathrm{m}}^c} \cong \mathcal{BPS}^{(0)}_{\mathrm{L}\bar{Y}}.
    \end{equation}
    Also, using the morphism $\mathrm{L} Y \to \mathbb{G}_{\mathrm{m}}^c$ induced by the global equivariant parameter and the Leray--Hirsch theorem, we obtain a natural isomorphism of functors
    \begin{equation}\label{eq-leray-hirsch}
     \bar{q}_* \bar{q}^* (-) \cong (-) \otimes \mathrm{H}^*( \mathbb{G}_{\mathrm{m}}^c).    
    \end{equation}
    Combining \eqref{eq-BPS-pullback}, \eqref{eq-BPS-push} and \eqref{eq-leray-hirsch}, we obtain the desired statement.
\end{proof}

\begin{remark}
    When $\CY$ is the moduli stack of objects in a certain $2$-Calabi--Yau category, Kaubrys \cite[Theorem 1.4]{kaubrys2025exponential} gives an explicit description of the BPS sheaf for $\mathrm{L} Y$, which implies
     \Cref{prop-compare-BPS} (for $\CY /\mathrm{B} \mathbb{G}_{\mathrm{m}}$) and \Cref{prop-mult-dim-red-general}. His method is different from ours and uses the exponential maps and derived analytic geometry.
     We note that our argument for \Cref{prop-mult-dim-red-general} gives an alternative proof to that statement.
\end{remark}

\section{Applications to \texorpdfstring{$G$}{$G$}-Higgs bundles}

In this section, we apply \Cref{thm-main-0} to the BPS cohomology of the moduli space of semistable $G$-Higgs bundles.
For this, we recall the notion of quasi-isolated elements from \cite{bonnafe2005quasi}.
Let $G$ be a connected semisimple group. A semisimple element $g \in G$ is called \emph{quasi-isolated} if the centralizer $\mathrm{C}_G(g)$ has finite center.
We let $G_{\mathrm{qis}} \subset G$ denote the subset of quasi-isolated elements.
It follows from \cite[Corollary 2.13]{bonnafe2005quasi} that there is an inclusion $G_{\mathrm{qis}}  \subset G_{\mathrm{tor}} $.
The number of conjugacy classes in $G_{\mathrm{qis}}$ is finite and there is an explicit classification: see \cite[Theorem 5.1]{bonnafe2005quasi}.
For example, when $G$ is simply connected, $G$-conjugacy classes of quasi-isolated elements bijectively correspond to vertices of the extended Dynkin diagram.

For a torsion element $g \in G$, consider the following diagram:
\[\begin{tikzcd}
	{\CH_{\mathrm{C}_G(g), \textnormal{$0$-type}}^{\mathrm{ss}}} & {\CH_G^{\mathrm{ss}}} & {\CL \CH_G^{\mathrm{ss}}} \\
	{\mathrm{H}_{\mathrm{C}_G(g), \textnormal{$0$-type}}^{\mathrm{ss}}} & {\mathrm{H}_G^{\mathrm{ss}}} & {\mathrm{L}\mathrm{H}_G^{\mathrm{ss}}.}
	\arrow["{\iota_g}", from=1-1, to=1-2]
	\arrow["{p_{g}}", from=1-1, to=2-1]
	\arrow["{p_G}", from=1-2, to=2-2]
	\arrow["{r_G}"', from=1-3, to=1-2]
	\arrow["{\tilde{p}_G}", from=1-3, to=2-3]
	\arrow["{\bar{\iota}_g}"', from=2-1, to=2-2]
	\arrow["{\bar{r}_G}", from=2-3, to=2-2]
\end{tikzcd}\]
Here, the vertical arrows are good moduli morphisms and the left horizontal arrows are well-defined thanks to \Cref{prop-torsionloop-HGss}.
We have the following statement as an application of multiplicative dimensional reduction.

\begin{theorem}\label{thm-mult-dim-red-HGss}
    There is a natural isomorphism
    \[
     \bar{r}_{G, *} \mathcal{BPS}^{(0)}_{\mathrm{LH}_{G}^{\mathrm{ss}}} \cong \bigoplus_{[g] \in G_{\mathrm{qis}} / G}  \bar{\iota}_{g, *} \mathcal{BPS}^{(0)}_{\mathrm{H}_{\mathrm{C}_G(g), \textnormal{$0$-type}}^{\mathrm{ss}}}.
    \]
\end{theorem}

\begin{proof}
    By \Cref{prop-torsionloop-HGss} and \Cref{thm-main-0}, we obtain
    \[
        \bar{r}_{G, *} \mathcal{BPS}^{(0)}_{\mathrm{LH}_{G}^{\mathrm{ss}}} \cong \bigoplus_{[g] \in G_{\mathrm{tor}} / G}  \bar{\iota}_{g, *} \mathcal{BPS}^{(0)}_{\mathrm{H}_{\mathrm{C}_G(g), \textnormal{$0$-type}}^{\mathrm{ss}}}.
    \]
    Therefore it suffices to show that for any $g \in G_{\mathrm{tor}} \setminus G_{\mathrm{qis}}$, the following vanishing holds:
    \[
        \mathcal{BPS}^{(0)}_{\mathrm{H}_{\mathrm{C}_G(g)}^{\mathrm{ss}}} = 0.
    \]
    Since $g$ is not quasi-isolated, $\mathrm{Z}(\mathrm{C}_G(g))$ contains a non-trivial torus.
    In particular, for each closed point $x \in \mathcal{H}_{\mathrm{C}_G(g)}^{\mathrm{ss}}$, there exists a non-trivial torus in the stabilizer group whose action on the tangent complex is trivial.
    Then the vanishing follows from \Cref{prop-supp-lem}.
\end{proof}


\begin{remark}\label{rem-general-reductive}
    For a general connected reductive group $G$, we can study the BPS cohomology for $\mathrm{LH}_{G}^{\mathrm{ss}}$ by reducing to the case of semisimple groups as follows.
    Let $T' \subset G$ be the neutral component of the center $\mathrm{Z}(G)$ and set $c \coloneqq \dim T'$ and $\bar{G} \coloneqq G /T'$.
    Take $d \in \pi_1(G)_{\mathbb{Q}}$ and $\bar{d} \in \pi_1(\bar{G})_{\mathbb{Q}}$ be its image. Consider the following commutative diagram:
    \[\begin{tikzcd}
        {\CL\CH_{G, d}^{\mathrm{ss}}} & {\CL\CH_{\bar{G}, \bar{d}}^{\mathrm{ss}}} \\
        {\mathrm{LH}_{G, d}^{\mathrm{ss}}} & {\mathrm{LH}_{\bar{G}, \bar{d}}^{\mathrm{ss}}.}
        \arrow["q", from=1-1, to=1-2]
        \arrow[from=1-1, to=2-1]
        \arrow[from=1-2, to=2-2]
        \arrow["{\bar{q}}"', from=2-1, to=2-2]
    \end{tikzcd}\]
    One easily sees that the map $q \colon \CL\CH_{G, d}^{\mathrm{ss}} \to \CL\CH_{\bar{G}, \bar{d}}^{\mathrm{ss}}$ is a principal $(\CL\mathrm{T}^* \mathrm{Pic}^0(C) \times \mathbb{G}_{\mathrm{m}} \times \mathrm{B}\mathbb{G}_{\mathrm{m}})^c$-bundle onto its image $(\CL\CH_{\bar{G}, \bar{d}}^{\mathrm{ss}})_{, \textnormal{liftable}} \subset \CL\CH_{\bar{G}, \bar{d}}^{\mathrm{ss}}$.
    Note that the abelianization $G \to G_{\mathrm{ab}}$ induces a map 
    \[
        \CL\CH_{G, d}^{\mathrm{ss}} \to (\CL \mathrm{T}^* \mathrm{Pic}^0(C) \times \mathbb{G}_{\mathrm{m}} \times \mathrm{B}\mathbb{G}_{\mathrm{m}})^c.
    \]
    Therefore we can use the Leray--Hirsch theorem to obtain
    \[
     q_* \varphi_{\CL\CH_{G, d}^{\mathrm{ss}}} \cong \varphi_{(\CL\CH_{\bar{G}, \bar{d}}^{\mathrm{ss}})_{\textnormal{liftable}}} \otimes \mathrm{H}^* \left( \mathrm{Pic}^0(C) \times \mathbb{G}_{\mathrm{m}} \times \mathrm{B}\mathbb{G}_{\mathrm{m}}\right)^{\otimes c}[2c \cdot \dim \mathrm{Pic}^0(C)].  
    \]
    Similarly, we have
    \[
     \bar{q}_* \mathcal{BPS}^{(c)}_{\mathrm{LH}_{G, d}^{\mathrm{ss}}} \cong \mathcal{BPS}^{(0)}_{(\mathrm{LH}_{\bar{G}, \bar{d}}^{\mathrm{ss}})_{\textnormal{liftable}}} \otimes  \mathrm{H}^* \left( \mathrm{Pic}^0(C) \times \mathbb{G}_{\mathrm{m}}\right)^{\otimes c}[2c \cdot \dim \mathrm{Pic}^0(C) + c]. 
    \]
    See the proof of \Cref{prop-mult-dim-red-general} for an analogous discussion.
\end{remark}

For semisimple and connected $G$, we define the \emph{stringy BPS cohomology} for $\mathrm{H}_G^\mathrm{ss}$ by
\[
 \mathrm{H}^{ *}_{\mathrm{BPS}, \mathrm{st}} (\mathrm{H}_G^\mathrm{ss}) \coloneqq \bigoplus_{[g] \in G_{\mathrm{qis}} / G} \mathrm{H}^* \left(\mathrm{H}^{\mathrm{ss}}_{\mathrm{C}_G(g), \textnormal{$0$-type}}, \mathcal{BPS}^{(0)}_{\mathrm{H}_{\mathrm{C}_G(g), \textnormal{$0$-type}}^{\mathrm{ss}}}\right).
\]
\Cref{thm-mult-dim-red-HGss} implies
\begin{equation}\label{eq-stringy-loop}
    \mathrm{H}^*(\mathrm{LH}_G^{\mathrm{ss}}, \mathcal{BPS}^{(0)}_{\mathrm{LH}_G^{\mathrm{ss}}}) \cong \mathrm{H}^{ *}_{\mathrm{BPS}, \mathrm{st}}(\mathrm{H}_G^\mathrm{ss}).
\end{equation}

In \cite[Conjecture 10.3.18]{bu2025cohomology}, we formulated the topological mirror symmetry conjecture for the moduli space of $G$-Higgs bundles using the BPS cohomology of the good moduli space of the loop stack.
The isomorphism \eqref{eq-stringy-loop} implies that the following formulation is equivalent to \cite[Conjecture 10.3.18]{bu2025cohomology}.

\begin{conjecture}\label{conj-stringy-BPS}
    Let $C$ be a smooth projective curve, $G$ be a connected semisimple group and $G^{\vee}$ its Langlands dual. Then there exists a natural isomorphism
    \[
        \mathrm{H}^{ *}_{\mathrm{BPS}, \mathrm{st}} (\mathrm{H}_G^\mathrm{ss}) \cong \mathrm{H}^{ *}_{\mathrm{BPS}, \mathrm{st}} (\mathrm{H}_{G^{\vee}}^\mathrm{ss}).  
    \]
\end{conjecture}

An advantage of using the stringy BPS cohomology rather than the BPS cohomology of the loop space is that one can reduce to the logarithmic case (i.e., $K_C(D)$-twisted case for $D > 0$) using the vanishing cycle trick as initiated by Maulik--Shen \cite{maulik2021endoscopic} and further investigated by the author and collaborators \cite{kinjo2024cohomological,kinjo2024global}.
We briefly explain the idea here. For simplicity, we will assume that $g(C) \geq 2$ from now on.

Let $\CL$ be a line bundle with $\deg \CL > 2 g(C) - 2$ and $\CH_{G, \CL}^{\mathrm{ss}}$ denote the moduli stack of $\CL$-twisted semistable $G$-Higgs bundles on $C$.
It is known that $\CH_{G, \CL}^{\mathrm{ss}}$ is smooth and admits a good moduli space $p_{G, \CL} \colon \CH_{G, \CL}^{\mathrm{ss}} \to \mathrm{H}_{G, \CL}^{\mathrm{ss}}$.
Assume that $G$ is semisimple and $g \in G$ be a semisimple element. We define
\[
    \overline{\CH}_{\mathrm{C}_G(g), \CL}^{\mathrm{st}} \subset \CH_{\mathrm{C}_G(g), \CL}^{\mathrm{ss}}
\]
the union of connected components which contains a point with finite stabilizer groups.
Let $\overline{\mathrm{H}}_{\mathrm{C}_G(g), \CL}^{\mathrm{st}} $ be the good moduli space and we define the stringy intersection cohomology as
\[
 \mathrm{IH}_{\mathrm{st}}^*(\mathrm{H}_{G, \mathcal{L}}^{\mathrm{ss}})  \coloneqq \bigoplus_{[g] \in G_{\mathrm{qis}} / G} \mathrm{IH}^* \left(\overline{\mathrm{H}}^{\mathrm{st}}_{\mathrm{C}_G(g), \CL,  \textnormal{$0$-type}} \right).
\]
We conjecture the following logarithmic version of the topological mirror symmetry.
\begin{conjecture}\label{conj-stringy-IH}
    Let $C$ be a smooth projective curve with $g(C)\geq 2$ and $G$ be a connected semisimple group. Then there exists a natural isomorphism
    \[
        \mathrm{IH}_{\mathrm{st}}^*(\mathrm{H}_{G, \CL}^\mathrm{ss}) \cong \mathrm{IH}_{\mathrm{st}}^*(\mathrm{H}_{G^{\vee}, \CL}^\mathrm{ss}).  
    \]
\end{conjecture}

\begin{remark}
    For $g \in G$, we let $\mathrm{C}_G(g)^{\circ} \subset \mathrm{C}_G(g)$ the neutral component and $\Gamma_g \coloneqq \pi_0(\mathrm{C}_G(g))$.
    Recall that a principal $\mathrm{C}_G(g)$-bundle on $C$ corresponds to a triple $(F, \rho, P)$ of
    a subgroup $F \subset \Gamma_g$, a connected principal $F$-bundle $\rho \colon C_{\rho} \to C$ and 
    an $F$-equivariant principal $\mathrm{C}_G(g)^{\circ}$-bundle $P$ on $C_{\rho}$, where $F$ acts on $\mathrm{C}_G(g)^{\circ}$ by the conjugation.
    We define an open and closed substack
    \[
        \CH_{\mathrm{C}_G(g), \CL, \mathrm{ell}}^{\mathrm{ss}} \subset \CH_{\mathrm{C}_G(g), \CL}^{\mathrm{ss}} 
    \]
    consisting of those $(F, \rho, P)$ with finite $\mathrm{Z}(\mathrm{C}_G(g)^{F})$. We expect an equality 
    \[
        \overline{\CH}^{\mathrm{st}}_{\mathrm{C}_G(g), \CL,  \textnormal{$0$-type}} = \CH_{\mathrm{C}_G(g), \CL, \mathrm{ell}}^{\mathrm{ss}} .
    \]
\end{remark}

We note that the decomposition theorem for the logarithmic Hitchin fibrations for general connected reductive groups was investigated by \textcite{de2025hitchin}, which is expected to have some applications to \Cref{conj-stringy-IH}.

We now explain that the arguments in \cite{kinjo2024cohomological,kinjo2024global} generalize to arbitrary reductive groups and a sheaf-theoretic version of \Cref{conj-stringy-IH} implies \Cref{conj-stringy-BPS}.
This will be explained in detail in forthcoming work of the author with Pavel Safronov, and we only provide some ideas here.
First, one can show that there is a function $f$ on $\mathcal{H}_{G, \CL}$ such that an equivalence of oriented $(-1)$-shifted symplectic stacks $\Crit(f) \cong \mathrm{T}^*[-1] \CH_{G}$ exists.
Further, $f$ descends to a function $f_B$ on the Hitchin base. Let $\bar{f}$ be the function on $\mathrm{H}_{G, \CL}^\mathrm{ss}$ induced from $f$.
Then one can show that the vanishing cycle functor with respect to $\bar{f}$ applied to the stringy intersection complex recovers the stringy BPS sheaf.
Then, assuming \Cref{conj-stringy-IH} holds as an isomorphism over Hitchin base, the vanishing cycle functor with respect to $f_B$ applied to this isomorphism yields \Cref{conj-stringy-BPS}.

\section{Applications to \texorpdfstring{$G$}{$G$}-local systems and its twisted version}

Let $\Sigma$ be a compact oriented $2$-manifold and $G$ be a semisimple group. 
For a torsion element $g \in G$, consider the following diagram:
\[\begin{tikzcd}
	{\mathcal{L}\mathrm{oc}_{\mathrm{C}_G(g)}(\Sigma)} & {\mathcal{L}\mathrm{oc}_{G}(\Sigma)} & {\mathcal{L}\mathrm{oc}_{G}(\Sigma \times S^1)} \\
	{\mathrm{Loc}_{\mathrm{C}_G(g)}(\Sigma)} & {\mathrm{Loc}_{G}(\Sigma)} & {\mathrm{Loc}_{G}(\Sigma \times S^1).}
	\arrow["{{\iota_g}}", from=1-1, to=1-2]
	\arrow["{{p_{g}}}", from=1-1, to=2-1]
	\arrow["{{p_G}}", from=1-2, to=2-2]
	\arrow["{{r_G}}"', from=1-3, to=1-2]
	\arrow["{{\tilde{p}_G}}", from=1-3, to=2-3]
	\arrow["{{\bar{\iota}_g}}"', from=2-1, to=2-2]
	\arrow["{{\bar{r}_G}}", from=2-3, to=2-2]
\end{tikzcd}\]
Here, the vertical arrows are good moduli morphisms.
Using \Cref{cor-torsionloop-HG} and arguing as in the proof of \Cref{thm-mult-dim-red-HGss}, we obtain the following:
\begin{theorem}\label{thm-mult-dimred-loc}
    There is a natural isomorphism
    \[
     \bar{r}_{G, *} \mathcal{BPS}^{(0)}_{\mathrm{Loc}_{G}(\Sigma \times S^1)} \cong \bigoplus_{[g] \in G_{\mathrm{qis}} / G}  \bar{\iota}_{g, *} \mathcal{BPS}^{(0)}_{\mathrm{LLoc}_{\mathrm{C}_G(g)}(\Sigma)}.
    \]
\end{theorem}

Now we discuss a twisted version of the multiplicative dimensional reduction, which applies to cohomological Donaldson--Thomas theory for $3$-manifolds that are non-trivial $S^1$-bundles.
Let $M$ be a Seifert-fibred $3$-manifold, i.e., there exists an $S^1$-fibration $M \to \Sigma_{\mathrm{orb}}$ to a compact oriented orbifold with finite stabilizer groups.
We have the following statement, which is a twisted version of \Cref{prop-support}:

\begin{proposition}\label{prop-support-twisted}
    Let $M$ be a Seifert-fibred $3$-manifold and $G$ be a semisimple group.
    Then, for $[E] \in \Supp \mathcal{BPS}^{(0)}_{\mathrm{Loc}_G(M)}$ represented by a semisimple $G$-local system $E$, the monodromy $[g_E] \in G / G$ corresponding to the fibre $S^1$ is a torsion element.
\end{proposition}

\begin{proof}
    Note that $\mathrm{Aut}(E)$ is reductive and $g_E$ lies in the center of $\mathrm{Aut}(E)$.
    In particular, $g_E$ is semisimple.
    Assume that $g_E$ is not a torsion element.
    Then, the Zariski closure of the set $\{ g_E^n \mid n\in \mathbb{Z}\}$ contains a one-dimensional torus $T' \subset \mathrm{Z}(\mathrm{Aut}(E))$.
    Using \Cref{prop-supp-lem}, it suffices to show that the $T'$-action on the tangent space at $[E] \in \mathcal{L}\mathrm{oc}_G(M)$ is trivial.
    In other words, it suffices to show that the action of $T'$ on $\mathrm{H}^*(M, \mathrm{Ad}(E))$ is trivial.

    To see this, consider the weight decomposition with respect to $T'$
    \[
        \mathrm{Ad}(E) = \bigoplus_{i \in \mathbb{Z}} \mathrm{Ad}(E)_i
    \]
    corresponding to $\mathfrak{g} = \bigoplus_{i \in \mathbb{Z}}  \mathfrak{g}_i$.
    It suffices to show that the cohomology of $\mathrm{Ad}(E)_i$ is zero if $i \neq 0$.
    We note that the monodromy operator 
    \[
     g_E \cdot \colon \mathfrak{g}_i \to \mathfrak{g}_i
    \]
    does not contain $1$ as an eigenvalue for $i \neq 0$. Indeed, an eigenvector for eigenvalue $1$ needs to be a fixed point with respect to the $T'$-action.
    Let $f \colon M \to \Sigma_{\mathrm{orb}}$ be the quotient map.
    Then, since the monodromy operator on $\mathrm{Ad}(E)_i$ does not contain a fixed part for $i \neq 0$, we have $f_* \mathrm{Ad}(E)_i = 0$.
    In particular, we have the vanishing of the cohomology as desired.
\end{proof}

For each positive integer $n \in \mathbb{Z}_{>0}$, consider the quotient
\[
 \Sigma_{\mathrm{orb}, M}^{[n]} \coloneqq M /_{\mu_n} S^1,    
\]
where $\mu_n$ is the natural action of $S^1$ on $M$ composed with the $n$-th multiplication map $\cdot n \colon S^1 \to S^1$.
In particular, $\Sigma_{\mathrm{orb}, M}^{[n]}$ is an orbifold whose generic stabilizers are $\mathbb{Z}/ n \mathbb{Z}$.
Consider the natural map
\[
 \iota_{n} \colon \mathcal{L}\mathrm{oc}_G(\Sigma_{\mathrm{orb}, M}^{[n]}) \to \mathcal{L}\mathrm{oc}_G(M), \quad \bar{\iota}_n \colon \mathrm{Loc}_{G}(\Sigma_{\mathrm{orb}, M}^{[n]}) \to \mathrm{L}\mathrm{oc}_G(M).
\]
\Cref{prop-support-twisted} and the quasi-compactness of the character stack implies the following:

\begin{corollary}\label{cor-enough-divisible}
    Let $M$ be a Seifert-fibred $3$-manifold and $G$ be a semisimple group.
    Then there exists an integer $N \in \mathbb{Z}_{>0}$ such that $\Supp \mathcal{BPS}^{(0)}_{\mathrm{Loc}_G(M)}$ is contained in the image of $\bar{\iota}_N$.
\end{corollary}

Moreover, we expect the following:
\begin{conjecture}\label{conj-Seifert}
    \begin{enumerate}
        \item The character stack $\mathcal{L}\mathrm{oc}_G(\Sigma_{\mathrm{orb}, M}^{[n]})$ admits a canonical $0$-shifted symplectic structure.
        \item The map $\iota_{n} \colon \mathcal{L}\mathrm{oc}_G(\Sigma_{\mathrm{orb}, M}^{[n]}) \to \mathcal{L}\mathrm{oc}_G(M)$ admits a natural Lagrangian structure.
        \item Take a sufficiently divisible $N$ as in \Cref{cor-enough-divisible}. Then there exists a natural isomorphism 
        \[
          \bar{\iota}_{N, *} \mathcal{BPS}_{\mathrm{Loc}_G(\Sigma_{\mathrm{orb}, M}^{[N]})}^{(0)} \cong \mathcal{BPS}_{\mathrm{Loc}_G(M)}^{(0)}.   
        \]
    \end{enumerate}
\end{conjecture}

We expect that the first two statements are consequences of the AKSZ formalism.
Assuming them, the deformation space of $\iota_n$ to the normal cone carries a relative exact $(-1)$-shifted symplectic structure.
Then, the isomorphism in the third statement would be constructed as the specialization map, as in the untwisted case in \Cref{thm-main-0}.
Since the proof of \Cref{prop-support-twisted} shows that the character stack $\mathcal{L}\mathrm{oc}_G(M)$ around a point in the support of the BPS sheaf is formal-locally modeled on the loop space,
one can show that the specialization map between the BPS sheaves is an isomorphism once it is constructed, by reducing to the untwisted case as proved in \Cref{prop-compare-BPS}.

    Let $\Sigma$ be a compact oriented $2$-manifold. 
    Then $S^1$-bundles on $\Sigma$ are classified by the first Chern class in $\mathrm{H}^2(\Sigma) \cong \mathbb{Z}$.
    We let $M_1$ be the $S^1$-bundle on $\Sigma$ corresponding to $1 \in \mathbb{Z}$.
    The character stack $\mathcal{L}\mathrm{oc}_G(\Sigma_{M_1}^{[n]})$ can be described as follows.
    First, it decomposes according to the monodromy of order $n$ along the fibre $S^1$:
    \[
        \mathcal{L}\mathrm{oc}_G(\Sigma_{M_1}^{[n]}) = \coprod_{\substack{[g] \in G / G \\ g^n = e}}  \mathcal{L}\mathrm{oc}_G(\Sigma_{M_1}^{[n]})_g.
    \]
    Let $\mathrm{N}_{G}(\langle g \rangle)$ be the normalizer of $\langle g \rangle \subset G$ and let $m$ be the order of the kernel of the map $\mathbb{Z} / n \mathbb{Z} \xrightarrow{1 \mapsto g} G$.
    Then $\mathcal{L}\mathrm{oc}_G(\Sigma_{M_1}^{[n]})_g$ is the stack parametrizing the following homotopy commutative diagram:
    \[\begin{tikzcd}
        \Sigma & {(\mathrm{B}\mathrm{N}_{G}(\langle g \rangle) / \langle g \rangle) \times \mathrm{B}(\mathbb{Z} / m \mathbb{Z})} \\
        {\mathrm{B}^2 (\mathbb{Z} / n \mathbb{Z} )} & {\mathrm{B}^2 (\mathbb{Z} / n \mathbb{Z} )}
        \arrow[from=1-1, to=1-2]
        \arrow[from=1-1, to=2-1]
        \arrow[from=1-2, to=2-2]
        \arrow[equals, from=2-1, to=2-2]
    \end{tikzcd}\]
    where the left vertical map corresponds to the generator in $\mathrm{H}^2(\Sigma, \mathbb{Z} / n \mathbb{Z}) $ and the right vertical map classifies $\mathrm{B} \mathrm{N}_{G}(\langle g \rangle) \to (\mathrm{B}\mathrm{N}_{G}(\langle g \rangle) / \langle g \rangle) \times \mathrm{B}(\mathbb{Z} / m \mathbb{Z})$.
    By expanding the definition, we see that there is a natural map
    \[
        \mathcal{L}\mathrm{oc}_G(\Sigma_{M_1}^{[n]})_g \to \mathcal{L}\mathrm{oc}_{\mathrm{N}_{G}(\langle g \rangle) / \langle g \rangle}(\Sigma)_1 
    \]
    where $\mathcal{L}\mathrm{oc}_{\mathrm{N}_{G}(\langle g \rangle) / \langle g \rangle}(\Sigma)_1$ is the open and closed substack of $\mathcal{L}\mathrm{oc}_{\mathrm{N}_{G}(\langle g \rangle) / \langle g \rangle}(\Sigma)$ whose associated $\mathbb{Z} / n \mathbb{Z}$-gerbe corresponds to the generator in $\mathrm{H}^2(C, \mathbb{Z} / n \mathbb{Z})$,
    and the above map is a $\mathcal{L}\mathrm{oc}_{\langle g \rangle}(\Sigma)$-torsor.
    In particular, $\mathcal{L}\mathrm{oc}_G(\Sigma_{M_1}^{[n]})_g$ carries a $0$-shifted symplectic structure.

    Motivated by the above discussion, we define the twisted stringy BPS cohomology for $\mathrm{Loc}_G(\Sigma)$ by
    \begin{equation}\label{eq-twisted-stringy}
     \mathrm{H}^*_{\mathrm{BPS}, \mathrm{st}, \mathrm{tw}}(\mathrm{Loc}_G(\Sigma)) \coloneqq \bigoplus_{\substack{[g] \in G / G \\ g^N = e}} \mathrm{H}^* \left(\mathrm{Loc}_G(\Sigma_{M_1}^{[N]})_g, \mathcal{BPS}^{(0)}_{\mathrm{Loc}_G(\Sigma_{M_1}^{[N]})_g} \right)
    \end{equation}
    for sufficiently divisible $N$.
    Assuming \Cref{conj-Seifert}, the twisted stringy BPS cohomology computes the BPS cohomology for $\mathrm{Loc}_G(M_1)$.
    In particular, the Langlands duality conjecture for the BPS cohomology of $3$-manifolds \cite[Conjecture 10.3.33]{bu2025cohomology} suggests that the twisted stringy BPS cohomology satisfies the Langlands duality conjecture:

\begin{conjecture}\label{conj-betti-twisted-stringy}
    For a compact oriented $2$-manifold $\Sigma$ and a semisimple group $G$, there exists a natural isomorphism
    \[
        \mathrm{H}^*_{\mathrm{BPS}, \mathrm{st}, \mathrm{tw}}(\mathrm{Loc}_G(\Sigma)) \cong \mathrm{H}^*_{\mathrm{BPS}, \mathrm{st}, \mathrm{tw}}(\mathrm{Loc}_{G^{\vee}}(\Sigma)).
    \]
\end{conjecture}

One can mimic the definition of the twisted stringy BPS cohomology for $\mathrm{Loc}_G(\Sigma)$ to define the 
twisted stringy BPS cohomology $\mathrm{H}^*_{\mathrm{BPS}, \mathrm{st}, \mathrm{tw}}(\mathrm{H}_G^{\mathrm{ss}})$ for the moduli space of semistable $G$-Higgs bundles.
We do not discuss the details here since it is parallel to the case of the character stacks and a similar discussion has already appeared in \cite[\S 10.3.24]{bu2025cohomology}.
We expect that the non-abelian Hodge correspondence preserves the twisted stringy BPS cohomology and that the Dolbeault version of \Cref{conj-betti-twisted-stringy} also holds:

\begin{conjecture}\label{conj-stringy-twisted-HGss}
    For a smooth projective curve $C$ and semisimple group $G$, there exists a natural isomorphism
    \[
        \mathrm{H}^*_{\mathrm{BPS}, \mathrm{st}, \mathrm{tw}}(\mathrm{H}^{\mathrm{ss}}_G) \cong \mathrm{H}^*_{\mathrm{BPS}, \mathrm{st}, \mathrm{tw}}(\mathrm{H}^{\mathrm{ss}}_{G^{\vee}}).
    \]
\end{conjecture}

\begin{remark}
    For a Seifert-fibred $3$-manifold $M \to \Sigma_{\mathrm{orb}}$, we believe that the moduli space $\mathcal{L}\mathrm{oc}_G(\Sigma_{\mathrm{orb}, M}^{[n]})$ is closely related to the moduli space of semistable parabolic Higgs bundles for a smooth projective curve with marked points in $\Sigma_{\mathrm{orb}}$.
    When $G =\mathrm{SL}(2, \mathbb{C})$, Muñoz-Echániz \cite[Theorem C]{munoz2025counting} has studied the relation between the stable loci of these moduli spaces and applied to the study of Abouzaid--Manolescu's $\mathrm{SL}(2, \mathbb{C})$-Floer homology \cite{abouzaid2020sheaf}, which coincides with the cohomology of the restriction of the BPS sheaf to the stable locus.

    This consideration motivates a parabolic version of \Cref{conj-stringy-twisted-HGss} corresponding to the Langlands duality for Seifert-fibred $3$-manifolds. We plan to investigate this further in the future work.
\end{remark}

\printbibliography

\end{document}